\documentclass{article}
\usepackage[a4paper,top=4.5cm,bottom=4.5cm,left=4cm,right=4cm]{geometry}

\usepackage{amsthm}
\usepackage{graphicx}%
\usepackage{multirow}%
\usepackage{amssymb}
\usepackage{mathrsfs}%
\usepackage[title]{appendix}%
\usepackage{xcolor}%
\usepackage{textcomp}%
\usepackage{manyfoot}%
\usepackage{booktabs}%
\usepackage{algorithm}%
\usepackage{algorithmicx}%
\usepackage{algpseudocode}%
\usepackage{listings}%
\usepackage{bm}%
\usepackage{listings}
\usepackage{yhmath}
\usepackage{mathtools}
\usepackage{mathrsfs}
\usepackage[utf8]{inputenc}
\newcommand{\essinf}{\textnormal{ess} \inf}
\newcommand{\esssup}{\textnormal{ess} \sup}
\newcommand{\diam}{\textnormal{diam}}
\newcommand{\rank}{\textnormal{rank}}
\newcommand{\Var}{\textnormal{Var}}
\newcommand{\muv}{\boldsymbol{\mu}}
\newtheorem{theorem}{Theorem}[section]
\newtheorem{proposition}{Proposition}
\newtheorem{definition}{Definition}
\newtheorem{assumption}{Assumption}
\newtheorem{remark}{Remark}
\newtheorem*{extheorem}{Theorem}
\newcommand{\nicola}[1]{#1}

\usepackage[colorlinks=true,linkcolor=black,anchorcolor=black,citecolor=black,filecolor=black,menucolor=black,runcolor=black,urlcolor=black]{hyperref} 
\usepackage{cleveref}
\usepackage{biblatex}
\bibliography{biblio}
\appto{\bibsetup}{\sloppy}

\title{Error estimates for POD-DL-ROMs: a deep learning framework for reduced order modeling of nonlinear parametrized PDEs enhanced by proper orthogonal decomposition}
\author{Simone Brivio$^1$, Stefania Fresca$^1$, Nicola R. Franco$^1$, Andrea Manzoni$^1$\vspace{4mm} \\ \normalsize{$^1$\textit{MOX - Dipartimento di Matematica, Politecnico di Milano,}}\vspace{-0.8mm} \\ \normalsize{\textit{P.zza Leonardo da Vinci 32, I-20133 Milano, Italy}}}
\date{}

\begin{document}

\maketitle

\begin{abstract}
POD-DL-ROMs have been recently proposed as an extremely versatile strategy to build accurate and reliable reduced order models (ROMs) for nonlinear parametrized partial differential equations, combining {\em (i)} a preliminary dimensionality reduction obtained through proper orthogonal decomposition (POD) for the sake of efficiency,  {\em (ii)} an autoencoder  architecture that further reduces the dimensionality of the POD space to a handful of latent coordinates, and {\em (iii)} a dense neural network to learn the map that describes the dynamics of the latent coordinates as a function of the input parameters and the time variable.  Within this work, we aim at justifying the outstanding approximation capabilities of POD-DL-ROMs by means of a thorough error analysis, showing how the sampling required to generate training data, the dimension of the POD space, and the complexity of the underlying neural networks, impact on the solution accuracy. This decomposition, combined with the constructive nature of the proofs, allows us to formulate practical criteria to control the relative error in the approximation of the solution field of interest, and derive general error estimates. Furthermore, we show that, from a theoretical point of view, POD-DL-ROMs outperform several deep learning-based techniques in terms of model complexity. Finally, we validate our findings by means of suitable numerical experiments, ranging from parameter-dependent operators analytically defined to several parametrized PDEs.
\end{abstract}

\begin{small}
\begin{flushleft} \textbf{Keywords} Operator Learning, Neural Networks, Approximation bounds, Reduced order modeling, parametrized PDEs, deep learning-based reduced order modeling
\end{flushleft}
\end{small}

\section{Introduction}\label{sec1}

Solutions to partial differential equations (PDEs) are not usually available in analytic form and need to be approximated by suitable \textit{high-fidelity} methods, such as the Finite Element Method (FEM) \cite{Quarteroni2017,Quarteroni2014}. The latter usually entails a suitable spatial discretization of the (bounded, compact) computational domain $\Omega \subset \mathbb{R}^d$, $d=1, 2, 3$, regulated by the step size $h>0$ and yielding a set of $N_h$ degrees of freedom, that in some cases might correspond to the vertices of the elements providing the domain discretization. 
\textit{High-fidelity} methods are usually referred to as full order models (FOMs) as they provide very accurate solutions, however resulting in computationally demanding strategies in terms of either time or resources. Within this work, we focus on a parametric setting, where in general the PDE solution $u$ depends not only on the spatial coordinate ${x} \in \Omega$ and the time variable $t \in \mathcal{T} = [0,T]$, but also on a parameter vector $\muv \in \mathcal{P}$ -- being the parameter space $\mathcal{P} \subset \mathbb{R}^{p}$ a compact set --  namely $u = u(x,\muv,t)$. Once the problem has been discretized in space, we aim at exploring the solution manifold $\mathcal{S}_{N_h} = \{\mathbf{u}(\muv,t) = [u(x_i,\muv,t)]_{i=1}^{N_h} \in \mathbb{R}^{N_h}: (\muv,t) \in \mathcal{P} \times \mathcal{T}\}$, evaluating the problem solution in multiple scenarios, for different parameter values. To carry out this task efficiently, as well as to tackle other {\em multi-query} tasks such as those involving Uncertainty Quantification and to perform real-time numerical simulations, FOMs must be replaced by efficient and reliable reduced order models (ROMs), a wide class of strategies providing very efficient results yet retaining an adequate representation of the solution manifold $\mathcal{S}_{N_h}$.

Linear projection-based ROMs, such as the reduced basis (RB) method relying on either greedy algorithms or the Proper Orthogonal Decomposition (POD) to build a low-dimensional linear trial subspace, are widely used in the context of parametrized PDEs. Usually relying on a (Petrov-)Galerkin projection to generate the corresponding ROM by enforcing at the reduced order level the physical constraints expressed by the FOM, these strategies feature however  several drawbacks, especially when dealing with time-dependent, nonlinear, and nonaffine problems, ultimately requiring suitable \textit{hyper-reduction} strategies such as the Empirical Interpolation Method (EIM) \cite{eim_2004,Farhat2020,Manzoni2016} or the Discrete EIM (DEIM, \cite{deim_2009}). Despite being very general, and widely applied, hyper-reduction techniques usually feature an intrusive nature, require to handle algebraic arrays extracted from the FOM, ultimately resulting in overwhelming computational costs when dealing with nonlinear time-dependent parametrized PDEs. 

To overcome these limitations, data-driven Deep Learning-based ROMs (DL-ROMs) were recently proposed in \cite{franco2023deep,dlrom_2021} and similar works \cite{lee_carlberg_2020,mucke_2021,Pant_2021,zhu_2018} to exploit the power of DNNs to both perform dimensionality reduction of a set of high-dimensional snapshots data (obtained by sampling the solution manifold) and learn parameter-to-solution maps nonintrusively. Unfortunately, these techniques require to train complex architectures and might become unfeasible to train as soon as the FOM dimension $N_h$ increases, suffering from the \textit{curse of dimensionality}  in their {\em vanilla} version. To counter this issue, POD-DL-ROMs were then introduced in \cite{pod_dl_rom_2022}, leveraging on the power of DL-ROMs and the physically-consistent dimensionality reduction achieved through POD, and then training a DL-ROM network using FOM data projected on a (possibly, large dimensional) POD space: overall, POD-DL-ROMs are capable of lower training efforts in terms of both memory storage and computational time. The POD-DL-ROM paradigm has been tested against several problems, showing remarkable approximation capabilities in the numerical simulation of, e.g., fluid flows and fluid-structure interaction problems \cite{POD-DL-ROM-fluids,pod_dl_rom_2022}, cardiac electrophysiology \cite{fresca2021pod}, and micro-electromechanical systems \cite{fresca2022deep} among others.  

However, a thorough numerical analysis of the POD-DL-ROM technique -- connecting, e.g., the complexity of the NN architectures involved in a POD-DL-ROM, the sampling error entailed by the selection of training data, the POD error generated while projecting those data onto a POD space, with the overall accuracy of the computed solution -- is still lacking.
%--  solutionaccuracy, and error bounds  questions about this novel technique are still open. For instance, how complex does the architecture need to be to reach a given accuracy? How much information do we lose with the initial dimensionality reduction? When is it convenient to use POD-DL-ROMs? Does the data sampling technique matter in the overall accuracy? 
Within this work, we aim at addressing these questions in light of a solid theoretical analysis, providing general error estimates for the POD-DL-ROM technique, assessing their validity in a series of numerical experiments involving different parametrized problems.

\subsection{Literature review and existing results}
\label{sec:1_1}

Thanks to the flourishing and rapidly evolving literature of Approximation Theory, many Deep Learning-based approaches to reduced order modeling are now being justified with rigorous theoretical results and error estimates. The majority of these are grounded on a notorious result by Yarotski (2017) \cite{yarotsky2017}, which we report below.
\begin{extheorem}[Yarotski \cite{yarotsky2017}]
    Let $b \in \mathbb{N}$, $b\ge1$ and $0<\varepsilon<1/2$. Any $f \in W^{s,+\infty}([0,1]^{b})$ can be approximated uniformly with an error of at most $\varepsilon$ by a ReLU Deep Neural Network (DNN) having at most $c \log(1/\varepsilon)$ layers and $c \varepsilon^{-b/s} \log(1/\varepsilon)$ weights, where $c = c(s,b,f)$ is a constant.
\end{extheorem}
Indeed, this result and its subsequent generalizations, see e.g. \cite{yarotsky_2018,guhring_2021}, constitute the foundation of many recent works, for instance:
\begin{itemize}
    \item[(i)] in \cite{dlrom_bounds}, the authors exploited these results to formulate an error analysis for general DL-ROMs. However, their analysis is limited to the time-independent case and does not resolve the \textit{curse of dimensionality}, as it binds the complexity of DL-ROMs linearly with the FOM dimension $N_{h}$;
    
    \item[(ii)] Yarotski's Theorem was also considered in \cite{appx_convenet_2022}, where the authors investigated the approximation capabilities of Convolutional Neural Networks (CNNs), suggesting a strong connection between these architectures and the Fourier transform;
    
    \item [(iii)] similarly, the results in \cite{yarotsky2017} are fundamental for the derivation of the approximation bounds reported in \cite{deeponets_bounds_2021}, which, instead, concern the DeepONet paradigm, an approach first proposed by Lu et al. in \cite{lu2021learning};

    \item[(iv)] finally, Yarotski's Theorem and its generalizations were also employed to derive approximation bounds for deep learning-based ROM strategies that couple POD and feedforward neural networks, see, e.g., \cite{bhattacharya_pod}.
\end{itemize}
Here, we aim at proposing a similar analysis for POD-DL-ROMs, emphasizing the main differences between this approach and the existing literature.

\subsection{Overall idea and paper structure}
\label{sec:1_2}
We analyze the overall approximation error entailed by the use of POD-DL-ROMs when dealing with the solution of both linear and nonlinear time-dependent parametrized PDEs by highlighting two separate error contributions: one, coming from the preliminary dimensionality reduction obtained through POD, and one entailed by the use of neural networks.

In brief, the idea goes as follows. First, we show that in the finite data regime, the overall error of a POD-DL-ROM, $\mathcal{E}_{R}$, can be decomposed as
$$\mathcal{E}_R \le \mathcal{E}_S + \mathcal{E}_{POD} + \mathcal{E}_{NN},$$
where $\mathcal{E}_S$ is the sampling error, $\mathcal{E}_{POD}$ is the POD projection error, and $\mathcal{E}_{NN}$ is the approximation error of the neural network model in the DL-ROM pipeline. Then, we address each of the three contributions separately. 

For the first two, we rely on classical arguments that bind together the discrete and the continuous formulation of POD, see e.g. \cite{Farhat2020,Manzoni2016}, ultimately showing that the sampling error vanishes as a function of the sample size, while $\mathcal{E}_{POD}$ is uniquely characterized by the eigenvalue decay of the data correlation matrix. In this sense, our analysis is strictly related to the one proposed in \cite{deeponets_bounds_2021}.
To study the neural network error, instead, we consider a specific construction that reflects the general philosophy of DL-ROM techniques. More precisely, we emphasize the fact that POD-DL-ROMs use a neural network architecture that is obtained through the combination of two networks: a feature map, $\phi$, which captures the roughness in the parameter-to-solution operator, and a smoother decoder $\phi$. In particular, we base our proof on a generalization of Yarotski's Theorem, due to Gühring et al. \cite{guhring_2021}, which, during the composition step, allows us to keep the approximation error under control. For the sake of better readability, we report the latter result below. 
\begin{extheorem}[Gühring et al. \cite{guhring_relu_2019}]
    Let $b,s \in \mathbb{N}$, with $b\ge1$, $s \ge 2$ and $n\in\{0,1\}$. For any tollerance $0 < \varepsilon < 1/2 $ and any $f \in W^{s,+\infty}([0,1]^{b})$, there exists  and a ReLU DNN $\phi$ having at most $c \log(1/\varepsilon)$ layers and $c \varepsilon^{-b/(n-s)} \log(1/\varepsilon)$ weights, where $c = c(s,b,f,n)$ is a constant, such that
    \begin{equation*}
        \|\phi - f\|_{W^{n,+\infty}([0,1]^{b})}<\varepsilon.
    \end{equation*}
\end{extheorem}

All of this ultimately allows us to characterize the accuracy of POD-DL-ROMs in terms of their complexity, providing explicit error bounds that we later compare with the existing literature and verify numerically.

The paper is organized as follows: in Section \ref{sec:three} we formulate the problem, describing rigorously the POD-DL-ROM approach and the reducibility measures for the framework at hand; Section~\ref{sec:four} contains the main results of this work, namely the \textit{error decomposition} formula, a \textit{lower bound} result and an \textit{upper bound} result for the approximation error. 
Section~\ref{sec:five} then demonstrates advantages of POD-DL-ROMs when compared to similar deep learning-based frameworks, such as, e.g., POD+DNN and DeepONets. Finally, a series of numerical experiments that validate the theoretical analysis is shown in Section~\ref{sec:six}, while the last section draws some conclusions and summarizes possible further developments.

\section{An overview of the POD-DL-ROM technique}\label{sec:three}

POD-DL-ROMs provide a general-purpose ROM approach combining a data dimensionality reduction obtained through POD with the DL-ROM approach. After introducing the general class of problems we deal with, we overview the main building blocks of the POD-DL-ROM technique. For further details regarding, e.g., detailed algorithms for the offline (or training) and the online query (or testing) stages, the interested reader can refer to, e.g., \cite{pod_dl_rom_2022}. An extension of the POD-DL-ROM technique in view of time forecasts of the problem solution out of the training time window has been proposed in \cite{fresca2021long}.

\subsection{Problem formulation} 
Within this work, we consider time-dependent parametric PDEs of the following type
\begin{equation}
    \label{eq:general_formulation}
    \left\{
    \begin{aligned}
        \frac{\partial u}{\partial t} + \mathcal{L} (\muv) u(\muv,t) + \mathcal{N}(u(\muv,t), \muv) &= f(\muv,t), & \mbox{in} \ \ \Omega \times (0,T] \\
        \mathcal{B}(\muv) u(\muv,t) &= g(\muv,t) , & \mbox{on} \ \ \partial{\Omega} \times (0,T] \\
        u(\muv,0) &= u_0(\muv), & \mbox{in} \ \ \Omega,
    \end{aligned}
    \right. 
\end{equation}
where:
\begin{itemize}
    \item $u = u(x,\muv,t)$ is the PDE solution $\forall x \in \Omega$. Here we highlight the explicit dependence of $u$ on the time variable $t \in \mathcal{T} = [0,T]$ (for some $T > 0$) and the input parameter vector $\muv \in \mathcal{P} \subset \mathbb{R}^{p}$, $\mathcal{P}$ compact;
    \item $\mathcal{L}$ is a linear operator, whereas $\mathcal{N}$ is a nonlinear operator and $\mathcal{B}$ is the boundary operator; virtually, all these operators might be $\muv$-dependent
    \item $u_0 = u_0(\muv)$ is the initial condition;
    \item $\Omega$ is the (bounded) spatial domain where the problem is set.
\end{itemize}
Depending on the nature of the problem, input parameter can refer to either physical or geometrical properties of the problem at hand. We considered the formulation \eqref{eq:general_formulation_FOM} as general framework since it describes a wide variety of problems ranging in the fields of engineering, physics, and life sciences, just to make a few examples. Introducing a computational mesh over $\Omega$ with mesh size $h>0$ and a corresponding space discretization of the problem \eqref{eq:general_formulation} having $N_h$ degrees of freedom (dofs) obtained through, e.g., the finite element method, the finite-dimensional counterpart of problem \eqref{eq:general_formulation} provides our FOM and reads as follows:
\begin{equation}
    \label{eq:general_formulation_FOM}
    \left\{
    \begin{aligned}
        {\bf M}(\muv) \frac{\partial {\bf u}}{\partial t}(\muv,t) + {\bf A}(\muv){\bf u} (\muv,t) + {\bf N}({\bf u}(\muv,t),\muv) &= {\bf f}(\muv,t), & t \in (0,T] \\
        {\bf u}(\muv, 0) &= {\bf u}_0(\muv), & 
    \end{aligned}
    \right. 
\end{equation}
where $\mathbf{u}(\muv,t) \in \mathbb{R}^{N_h}$ denotes the vector of the $N_h$ dofs of the FOM solution, ${\bf M}(\muv) \in \mathbb{R}^{N_h \times N_h}$ the mass matrix, ${\bf A}(\muv) \in \mathbb{R}^{N_h \times N_h}$ the stiffness matrix, ${\bf N}( \cdot, \muv): \mathbb{R}^{N_h} \rightarrow \mathbb{R}^{N_h}$ a nonlinear map,  ${\bf f}(\muv,t) \in \mathbb{R}^{N_h}$ the source term and ${\bf u}_0(\muv) \in \mathbb{R}^{N_h}$ the initial data.  The FOM \eqref{eq:general_formulation_FOM} is then discretized in time, introducing a suitable time advancing scheme over a partition of $\mathcal{T}$ made by $N_t$ time steps $\{t_k\}_{k=1}^{N_t}$. 

To explore efficiently the solution manifold $\mathcal{S}_{N_h} = \{\mathbf{u}(\muv,t)  \, : \,  (\muv,t) \in \mathcal{P} \times \mathcal{T}\}$ we employ the POD-DL-ROM technique, performing a two-step dimensionality reduction: first, POD (realized through randomized SVD) is applied on a set of FOM snapshots; then, a DL-ROM is built to approximate the map between $(\muv,t)$ and the POD generalized coordinates. This latter task can be achieved by relying on two neural network architectures, {\em (i)} a deep autoencoder -- possibly involving convolutional layers --  that extracts a set of few, latent coordinates, ultimately representing the reduced-order coordinates of the ROM, and {\em (ii)} a deep feedforward neural network, to learn the map between $(\muv,t)$ and these latent coordinates. Below, we report the main building blocks of a POD-DL-ROM, originally proposed in \cite{pod_dl_rom_2022}: \smallskip

\begin{itemize}
    \item [(i)] the snapshot matrix for the parameter vectors $\muv_j$, $j=1, \ldots ,N_s$ is collected, thus obtaining $\mathbf{U}_{j} = [\mathbf{u}(\muv_j,t_k)]_{k=1}^{N_t} \in \mathbb{R}^{N_h \times N_t}$; \smallskip
    
    \item [(ii)] the whole snapshot matrix is obtained stacking $\mathbf{U}_j$, $j=1,\ldots,N_s$, namely $\mathbf{U} = [\mathbf{U}_{j}]_{j=1}^{N_s}  \in \mathbb{R}^{N_h \times N_{data}}$, where $N_{data} = N_s N_t$; \smallskip
    
    \item [(iii)] a singular value decomposition (SVD) is performed on the snapshot matrix $\mathbf{U}$, and the first $N$ left singular vectors are retained, thus yielding  $\mathbf{U} \approx \mathbf{V} \bm{\Sigma} \mathbf{W}^T$, where $\mathbf{V} \in \mathbb{R}^{N_h \times N}$, $\bm{\Sigma} \in \mathbb{R}^{N \times N}$ and $\mathbf{W} \in \mathbb{R}^{N \times N_{data}}$. Then, projecting $\mathbf{U}$ on the reduced linear subspace $\mathbf{V} \in \mathbb{R}^{N_h \times N}$, we obtain a snapshot matrix for the POD coefficients $\mathbf{Q} = \mathbf{V}^T \mathbf{U}$; \smallskip
    
    \item [(iv)] the POD coefficient vectors $\bm{q}(\muv_j,t_k)$, $j=1,\ldots,N_s$, $k=1,\ldots,N_t$, obtained from the columns of $\mathbf{Q}$, along with the parameters vector $\muv_j$ and the time instants $t_k$, are used to train a DL-ROM. This latter consists of a deep autoencoder $\Psi \circ \Psi'$ and a deep feedforward neural network (to which we refer to as {\em reduced network}) $\phi$, defined as follows: 
    \begin{equation*}
        \label{eq:POD-DL-ROM-architecture}
        \left\{
        \begin{aligned}
         \bm{z}^{DYN} &= \phi(\bm{\theta}_{DYN};\muv_j,t_k)\\
                \bm{z}^{ENC} &= \Psi'(\bm{\theta}_{ENC};\bm{q}(\muv_j,t_k))\\
                \hat{\bm{q}} &= \Psi(\bm{\theta}_{DEC};\bm{z}^{DYN}(\bm{\theta}_{DYN},\muv_j,t_k)),
        \end{aligned}
        \right. 
    \end{equation*} 
    where $\phi,\Psi',\Psi$ are the reduced network, the encoder and the decoder, respectively, while $\bm{\theta}_{DYN}, \bm{\theta}_{ENC}, \bm{\theta}_{DEC}$ are their corresponding neural network weights and biases (they are omitted, hereon, for the sake of readability). 
    The three networks are trained according to the \textit{per-example} loss function below,
    \begin{equation*}
        \label{eq:POD-DL-ROM-loss}
         \mathcal{L}_{supervised} = \omega_{N} \mathcal{L}_{N} + \omega_n \mathcal{L}_{n},
    \end{equation*}
    where 
    \begin{equation*}
        \begin{aligned}
             \mathcal{L}_{N} &= \sum_{j=1}^{N_s}
             \sum_{k=1}^{N_t}
             \|\hat{\bm{q}}(\muv_j,t_k) - \bm{q}(\muv_j,t_k)\|^2, \\ \mathcal{L}_{n} &=  
             \sum_{j=1}^{N_s}
             \sum_{k=1}^{N_t}
             \|\bm{z}^{ENC}(\muv_j,t_k) - \bm{z}^{DYN}(\muv_j,t_k)\|^2,
        \end{aligned}
    \end{equation*}
    and $n$ denotes the latent dimension of the architecture. As a matter of notation, from hereon we equip any finite dimensional space $\mathbb{R}^b$ (for some $b\in \mathbb{N}$) with the $\ell^2$ norm: thus, unless otherwise stated, we define $\|\cdot\| := \|\cdot\|_2$.
    It is worth to remark that $\mathcal{L}_{N}$ penalizes high reconstruction errors and $\mathcal{L}_{n}$ ensures a good representation in the latent space. 
\end{itemize}

Recalling that $(\muv,t) \rightarrow \mathbf{V}\hat{\bm{q}}(\muv,t) \approx \mathbf{u}(\muv,t)$ provides the POD-DL-ROM approximation, the objective of the present work is to characterize the relative error
\begin{equation*}
       \mathcal{E}_R := \biggl(\int_{\mathcal{P} \times \mathcal{T}}  \frac{\|\mathbf{u}(\muv,t)- \mathbf{V}\hat{\bm{q}}(\muv,t)\|^2}{\|\mathbf{u}(\muv,t)\|^2} d(\muv,t) \biggl)^{1/2}
\end{equation*}
in terms of the POD-DL-ROMs complexity. Here, we choose to focus on analyzing $\mathcal{E}_R$ since it is a common measure for the accuracy in the ROM literature. Moreover, we highlight that the entire workflow yielding the error estimate we propose in this work is only based on the approximation error, without considering the contribution carried by the training error. The extension to more general vector energy norms including the contribution of symmetric positive definite mass matrices to define the counterpart of norms in functional spaces like, e.g., $L^2(\Omega)$ or $H^1(\Omega)$, is also straightforward and is not considered here for the sake of simplicity.

\subsection{POD: from the discrete to the continuous formulation} 
\label{subsec:POD-analysis}
Before proceeding towards the thorough analysis of $\mathcal{E}_R$, we have to appropriately define the working setting, which depends on the linear dimensionality reduction. First, we notice that even though within the POD-DL-ROM pipeline we computed the POD matrix $\mathbf{V}$ through the (randomized) SVD algorithm, thus using a fully \textit{data-driven} procedure that employs a set of training data, the relative error $\mathcal{E}_R$ aims at measuring the \textit{approximation} capabilities over the entire time-parameter space $\mathcal{P} \times \mathcal{T}$, taking advantage of a continuous formulation. Within this section, we aim at filling the gap between the discrete and the continuous formulation of POD, highlighting links and bounds, focusing initially only on the source of error coming from the projection phase, rather than directly considering $\mathcal{E}_R$: this allows us to set the ground upon which the more complex approximation results of POD-DL-ROM are based.

We start by considering the $(\mathcal{P} \times \mathcal{T})$-discrete setting, and the fact that $\mathbf{V}$ results from the solution of a minimization problem; indeed, denoting by 
\begin{equation*}
    \mathbf{K} = \frac{|\mathcal{P} \times \mathcal{T}|}{N_{data}}\mathbf{U}\mathbf{U}^T \in \mathbb{R}^{N_h \times N_h}
\end{equation*}
the (discrete) correlation matrix and by $\sigma_k^2$ its eigenvalues, it holds that \cite{Manzoni2016} 
\begin{equation*}
    \begin{aligned}
        \sum_{k>N} \sigma_k^2 &= \frac{|\mathcal{P} \times \mathcal{T}|}{N_{data}} \sum_{j=1}^{N_{data}}\|\mathbf{u}_j - \mathbf{V}\mathbf{V}^T\mathbf{u}_j\|^2 \\ &= \min_{\mathbf{W} \in \mathbb{R}^{N_h \times N}: \mathbf{W}^T\mathbf{W} = \bm{I}} \frac{|\mathcal{P} \times \mathcal{T}|}{N_{data}} \sum_{j=1}^{N_{data}}\|\mathbf{u}_j - \mathbf{W}\mathbf{W}^T\mathbf{u}_j\|^2,
    \end{aligned}
\end{equation*}
where $N$ is the chosen POD dimension and $\mathbf{u}_j$ is the solution vector that corresponds to the tuple $(\muv,t)_j$. 
We can proceed analogously for the $(\mathcal{P} \times \mathcal{T})$-continuous setting, by considering 
\begin{equation}\label{eq:correlation_matrix}
    \mathbf{K}_{\infty} = \int_{\mathcal{P} \times \mathcal{T}} \mathbf{u}(\muv,t) \mathbf{u}(\muv,t)^T d(\muv,t) \in \mathbb{R}^{N_h \times N_h}
\end{equation} 
as the (continuous) correlation matrix and denoting by $\sigma_{k,\infty}^2$ its eigenvalues; similarly, we can prove that there exists an optimal rank-$N$ matrix $\mathbf{V}_{\infty} \in \mathbb{R}^{N_h \times N}$ such that
\begin{equation*}
    \begin{aligned}
        \sum_{k>N} \sigma_{k,\infty}^2 &= \int_{\mathcal{P} \times \mathcal{T}} \|\mathbf{u}(\muv, t) - \mathbf{V}_{\infty}\mathbf{V}_{\infty}^T\mathbf{u}(\muv, t)\|^2  d(\muv,t) 
        \\ &= \min_{\mathbf{W} \in \mathbb{R}^{N_h \times N}: \mathbf{W}^T\mathbf{W} = \bm{I}} \int_{\mathcal{P} \times \mathcal{T}} \|\mathbf{u}(\muv, t) - \mathbf{W}\mathbf{W}^T\mathbf{u}(\muv, t)\|^2  d(\muv,t) 
    \end{aligned}
\end{equation*}

From the considerations above, we can infer that 
\begin{equation*}
    \begin{aligned}
         \sum_{k>N} \sigma_{k,\infty}^2 &= \int_{\mathcal{P} \times \mathcal{T}} \|\mathbf{u}(\muv, t) - \mathbf{V}_{\infty}\mathbf{V}_{\infty}^T\mathbf{u}(\muv, t)\|^2  d(\muv,t) 
         \\ & \le \int_{\mathcal{P} \times \mathcal{T}} \|\mathbf{u}(\muv, t) - \mathbf{V}\mathbf{V}^T\mathbf{u}(\muv, t)\|^2  d(\muv,t);
    \end{aligned}
\end{equation*}
from the inequality above, we can remark that the \textit{data-driven} POD matrix $\mathbf{V}$ is not optimal for the continuous formulation, which stems from the hypothesis of having infinite data samples, while being the best orthogonal matrix in terms of explained variability with respect the training data at hand. In other words,  even though $\mathbf{V}$ is optimal for the training data, we have no guarantee that it is optimal for the test data, too; however, since in practice we are not able to obtain the matrix $\mathbf{V}_{\infty}$,  we must necessarily rely on $\mathbf{V}$ also in the online testing phase. 

Finally, we show how the discrete and the continuous POD formulations are related: indeed, denoting by $[\cdot]_i$ the $i$-th entry of a vector, and extending this notation to matrices, we have that  $\forall k,l=1,\ldots,N_h$
\begin{equation*}
    [\mathbf{K}_{\infty} - \mathbf{K}]_{kl} = \int_{\mathcal{P} \times \mathcal{T}} [\mathbf{u}]_k [\mathbf{u}]_l d(\muv,t) - \frac{|\mathcal{P} \times \mathcal{T}|}{N_{data}} \sum_{j=1}^{N_{data}} [\mathbf{u}_j]_k [\mathbf{u}_j]_l,
\end{equation*}
recalling that $\mathbf{u}_j$ is the solution vector that corresponds to the tuple $(\muv,t)_j$. Upon requiring integrability (easily verified for non-trivial bounded solutions), we can use the Strong Law of Large Numbers \cite{jacod_protter} and obtain $[\mathbf{K} - \mathbf{K}_{\infty}]_{kl} \xrightarrow{a.s.} 0$ as $N_s, N_t \rightarrow \infty, \forall k,l=1,\ldots,N_h$, which implies that $\|\mathbf{K} - \mathbf{K}_{\infty}\|_1 \xrightarrow{a.s.} 0$, being $\|\mathbf{Z}\|_1$ any 1-norm of the squared matrix $\mathbf{Z}$. By employing Bauer-Fike's theorem \cite{Quarteroni2014} with the $1$-norm, we can state that, upon ordering, for any $\sigma_{k,\infty}^2$, there exists $\sigma_k^2$ belonging to the spectrum of $\mathbf{K}$ such that
\begin{equation*}
    |\sigma_k^2 - \sigma_{k,\infty}^2| \le K_1(\mathbf{X})\|\mathbf{K} - \mathbf{K}_{\infty}\|_1, \qquad \forall k=1,\ldots,N_h
\end{equation*}
where $\mathbf{X}$ is the matrix collecting the right eigenvectors of $\mathbf{K}$, and $K_1(\mathbf{X})$ denotes its condition number. Thus, we can conclude that, setting $N$ as the POD dimension, it holds that
\begin{equation*}
    \sum_{k>N} \sigma_k^2 \xrightarrow{a.s.} \sum_{k>N} \sigma_{k,\infty}^2, \qquad N_s, N_t \rightarrow \infty. 
\end{equation*}

\subsection{An overlook over the reducibility measures for POD-DL-ROMs} 

POD-DL-ROMs couple POD, for the sake of a preliminary dimensionality reduction, with an autoencoder-based architecture to reconstruct the parameter-to-POD-coefficients map. Thus, at first it is evident that the projection-based nature of the paradigm invokes the definition of a linear reducibility measure to account for the {\em FOM-to-POD} dimensionality reduction task. 
 
\begin{definition}
    Let $\mathcal{S}_{N_h} = \{\mathbf{u}(\muv,t) \in \mathbb{R}^{N_h}: (\muv,t) \in \mathcal{P} \times \mathcal{T}\}$ be the solution manifold. The linear Kolmogorov N-width of $\mathcal{S}_{N_h}$ is defined as
    \begin{equation*}
        d_N(\mathcal{S}_{N_h}) = \inf_{V_N \subset \mathbb{R}^{N_h}: \dim(V_N) = N} \sup_{\mathbf{u} \in \mathcal{S}_{N_h}} \inf_{\mathbf{v} \in V_N} \|\mathbf{u} - \mathbf{v}\|.
    \end{equation*}
\end{definition}

It is worth to notice that the linear Kolmogorov $N$-width is strictly related to the eigenvalues decay of the correlation matrix $\mathbf{K}_{\infty} \in \mathbb{R}^{N_h \times N_h}$. In fact, following the same notation of Subsection \ref{subsec:POD-analysis}, we have that:
\begin{equation*}
    \begin{aligned}
    \sqrt{\sum_{k>N} \sigma_{k,\infty}^2} &= \biggl(\int_{\mathcal{P} \times \mathcal{T}} \|\mathbf{u}(\muv,t) - \mathbf{V}_{\infty}\mathbf{V}_{\infty}^T\mathbf{u}(\muv,t)\|^2 d(\muv,t) \biggl)^{1/2}
    \\ & \le \biggl(\int_{\mathcal{P} \times \mathcal{T}} \|\mathbf{u}(\muv,t) - \mathbf{W}\mathbf{W}^T\mathbf{u}(\muv,t)\|^2 d(\muv,t) \biggl)^{1/2}
    \\ & \le  |\mathcal{P} \times \mathcal{T}|^{1/2} \sup_{(\muv,t) \in \mathcal{P} \times \mathcal{T}} \|\mathbf{u}(\muv,t) - \mathbf{W}\mathbf{W}^T\mathbf{u}(\muv,t)\|^2
    \end{aligned}
\end{equation*}
for any $\mathbf{W} \in \mathbb{R}^{N_h \times N}$; thus,
\begin{equation*}
    \sqrt{\sum_{k>N} \sigma_{k,\infty}^2} \le |\mathcal{P} \times \mathcal{T}|^{1/2} d_N(\mathcal{S}_{N_h}).
\end{equation*}
The above relationship shows that the eigenvalue decay is an alternative (and more practical) measure of reducibility, with respect to a weaker norm. However, notice that in practice we can only approximate the quantity $\sum_{k>N} \sigma_{k,\infty}^2 \approx  \sum_{k>N} \sigma_{k}^2$, which is consistent with the theory thanks to the convergence result presented in Subsection \ref{subsec:POD-analysis}.

The autoencoder-based architecture of a POD-DL-ROM introduces a second level of dimensionality reduction, which operates a further compression of the information coming from the parameter-to-POD-coefficients map $\mathcal{Q}: (\muv,t) \rightarrow \mathbf{V}^T\mathbf{u}(\muv,t)$. The nonlinear nature of the dimensionality reduction performed through the autoencoder $\Psi \circ \Psi'$ (being $\Psi', \Psi$ the \textit{encoder} and  the \textit{decoder}, respectively) induces a nonlinear analogue of the Kolmogorov $n$-width  \cite{nonlinear_kolmogorov}.

\begin{definition}
\label{def:nonlinear_kolmogorov}
The nonlinear Kolmogorov $n$-width of the reduced manifold $\mathcal{S}_N = \{\bm{q}(\muv,t) = \mathbf{V}^T\mathbf{u}(\muv,t) \in \mathbb{R}^{N}: (\muv,t) \in \mathcal{P} \times \mathcal{T}\}$ is defined as 
    \begin{equation*}
        \label{eq:nonlinear_kolmogorov}
        \delta_n(\mathcal{S}_N) = \inf_{\substack{\Psi \in C(\mathbb{R}^{N},\mathbb{R}^n) \\ \Psi' \in C(\mathbb{R}^n, \mathbb{R}^{N})}} \sup_{\mathbf{u} \in \mathcal{S}_{N_h}} \|\mathbf{u} - \Psi(\Psi'(\mathbf{u}))\|.
    \end{equation*}
\end{definition}

Now, to deal with nonlinear approximation methods, we state another fundamental definition upon which the main results of this work are based.
 
\begin{definition}
\label{def:perfect_embedding}
    The reduced manifold $\mathcal{S}_N = \{\bm{q}(\muv,t) = \mathbf{V}^T\mathbf{u}(\muv,t) \in \mathbb{R}^{N}: (\muv,t) \in \mathcal{P} \times \mathcal{T}\}$ enjoys the perfect embedding Assumption with regularity $s,s'$ if the infimum in Definition \ref{def:nonlinear_kolmogorov} is attained, namely there exist $\Psi_* \in C^s(\mathbb{R}^{N},\mathbb{R}^n), \Psi'_* \in C^{s'}(\mathbb{R}^n, \mathbb{R}^{N})$ such that
    \begin{equation*}
        \Psi_*(\Psi'_*(\bm{q}(\muv,t)) = \bm{q}(\muv,t) \qquad \forall (\muv,t) \in \mathcal{P} \times \mathcal{T}.
    \end{equation*}
\end{definition}
 
In conclusion, as we did with the POD dimension $N$, we need to characterize the latent dimension $n$ with a practical criterion. To do that, an extension of \textit{Theorem 3} provided in \cite{dlrom_bounds} shows that if the parameter-to-solution map $\mathcal{G}: (\muv,t) \rightarrow \mathbf{u}(\muv,t)$ and thus the parameter-to-POD-coefficients map $\mathcal{Q}: (\muv,t) \rightarrow \mathbf{V}^T\mathbf{u}(\muv,t)$ are Lipschitz-continuous, there exists $n \le 2p + 3$ such that $\delta_n(\mathcal{S}_N) = 0$.

\section{Main results}
\label{sec:four}
Before stating the main result of this work, namely an \textit{upper bound} result, that concerns only POD-DL-ROMs, we make some preliminary reasoning that, instead, applies to any POD+DNN approach, i.e. we do not constrain the neural network $\hat{\bm{q}}$, that approximates the parameter-to-POD-coefficients map, to be a DL-ROM. For this purpose, we briefly recall that the POD+DNN technique involves the reconstruction of the parameter-to-solution map through the approximation $(\muv,t) \mapsto \mathbf{V}\hat{\bm{q}}(\muv,t) \approx \mathbf{u}(\muv,t)$, where $\hat{\bm{q}}$ is a generic (possibly dense) neural network.

In particular, we start by characterizing $\mathcal{E}_R$ through an error decomposition formula, that enables us to describe the various error contributions and formulate possible strategies to control them. Secondly, we state a \textit{lower bound} result, that highlights how, regardless of the architecture of neural network $\hat{\bm{q}}$, the relative error $\mathcal{E}_R$ can be bounded from below by a quantity depending on the POD projection.
%We then 
\nicola{Then,} we move to our \textit{upper bound} result, where we quantify how complex a POD-DL-ROM should be in order to achieve a specific bound on the relative error $\mathcal{E}_R$.

%The entire theoretical framework developed in the present Section requires an appropriate set of assumptions which are supposed to hold hereon. 
We initially remark that the computation of the error $\mathcal{E}_R$ and other related quantities hinges upon the evaluation of complex integrals, possibly in high dimensional spaces, which can be effectively handled through Monte Carlo methods.
\nicola{In this respect, we shall make the following assumptions, which we assume to hold true hereon.}
%Indeed, when dealing with time-dependent FOM solutions obtained by means of a numerical solver, we employ the following setting.

\begin{assumption}[%Time-parameter space and s
\nicola{S}ampling criterion]
\label{assumption:sampling}
 \nicola{Let $p>0$,} assume that $\mathcal{P} \subset \mathbb{R}^p$ is %a 
 compact %subset, $p \ge 0$ 
 and denote $\mathcal{T} = [0,T]$ for \nicola{some} $T > 0$. \nicola{We assume that the training (and testing) snapshots are sampled uniformly and iid in the parameter space, }%Then, we sample 
 $\muv \sim \mathcal{U}(\mathcal{P})$, %iid,
 \nicola{while a uniform grid is employed for} %while 
 the time variable, $t \in \{\Delta t, 2 \Delta t, ..., N_t \Delta t\}$, where $N_t \in \mathbb{N}_{\ge 2}$ and $\Delta t = T / N_t$.
\end{assumption}
%Thus, we prove below an error estimates for the computation of the integrals. 

%Finally, the proofs of the main results present in the current Section require a set of hypotheses on the parameter-to-solution map. 

\begin{assumption}[Parameter-to-solution map]
\label{assumption:solution-map}
      %Under the hypotheses of Assumption \ref{assumption:sampling}, 
      \nicola{L}et $\mathcal{G}:\nicola{\mathcal{P} \times \mathcal{T}\to\mathbb{R}^{N_{h}}}$ %for any $(\muv,t) \in \mathcal{P} \times \mathcal{T}$ 
      be the parameter-to-solution map\nicola{, mapping $(\muv,t) \mapsto \mathbf{u}(\muv,t)$}. We %always
      assume that
      %true:
    \begin{itemize}
        \item[\nicola{i)}] %the boundedness hypothesis, namely
        $m = \essinf_{(\muv,t) \in \mathcal{P} \times \mathcal{T}} \|\mathbf{u}(\muv,t)\| > 0$, $M = \esssup_{(\muv,t) \in \mathcal{P} \times \mathcal{T}} \|\mathbf{u}(\muv,t)\| < \infty$;
        \nicola{\vspace{0.5em}}
        \item[\nicola{ii)}] %that 
        $\mathcal{G}$ is Lipschitz-continuous with constant $L >0$.
    \end{itemize}
\end{assumption}

\nicola{From these assumptions, one can easily derive a couple of auxiliary results, which will be of practical interest in the remainder, and are reported below; for the sake of brevity, their proofs are postponed to  Appendix  \ref{sec:appendix}}.

\begin{proposition}
\label{prop:monte-carlo}
    \nicola{Let $f\in L^2(\mathcal{P} \times \mathcal{T})$.} Under %the hypotheses of 
    Assumption \ref{assumption:sampling}, %, if $f=f(\muv,t)$ is bounded in $L^2(\mathcal{P} \times \mathcal{T})$, then
    one has
    \begin{equation*}
        \mathbb{E}\biggl| \int_{\mathcal{P} \times \mathcal{T}} f(\muv,t) d(\muv,t) - \frac{|\mathcal{P} \times \mathcal{T}|}{N_{data}} \sum_{i=1}^{N_t}\sum_{j=1}^{N_s} f(\muv_j, t_i) \biggr| \le O(N_s^{-1/2} + N_t^{-1}).
    \end{equation*}
    \nicola{where the expectation is taken across all the possible realizations of the data sampling procedure.} %$\{(\mu_{j}, t_{i})\}_{i,j}$.}
\end{proposition}

%\begin{proof}
% We refer the reader to Appendix \ref{sec:appendix}.
%\end{proof}

\begin{proposition}
\label{prop:norm}
    Under the Assumption \ref{assumption:solution-map}, define $w(\muv,t) = \|\mathbf{u}(\muv,t)\|^{-2}$. Then
    \begin{equation*}
        \|\cdot\|_{L^2_w} = \biggl(\int_{\mathcal{P} \times \mathcal{T}} \|\cdot\|^2 w(\muv,t) d(\muv,t) \biggr)^{1/2}
    \end{equation*}
    is a norm in $L^2(\mathcal{P}\times \mathcal{T}; \mathbb{R}^{N_h})$.   
\end{proposition}

%\begin{proof}
%    We refer the reader to Appendix \ref{sec:appendix}.
%\end{proof}

\subsection{The error decomposition formula}
\label{subsec:error_decomposition}
%We firstly demonstrate that $\mathcal{E}_R$ induces a norm in the space $L^2(\mathcal{P}\times \mathcal{T}; \mathbb{R}^{N_h})$, which is necessary in the  proof of the error decomposition formula.
% 
In the following, we state an error decomposition formula that is valid for any POD+DNN approach -- and, in particular, for our POD-DL-ROM strategy. Given the more general nature of this result, its formulation is therefore not restricted to the technique at hand.

\begin{theorem}
    \label{thm:error_decomposition}
      Let $\mathcal{G}: (\muv,t) \mapsto \mathbf{u}(\muv,t)$ for any $(\muv,t) \in \mathcal{P} \times \mathcal{T}$ be the parameter-to-solution map. Consider a POD+DNN approximation of $\mathcal{G}$ as $\mathcal{G}(\muv,t) \approx \mathbf{V}\hat{\bm{q}}$, where $\hat{\bm{q}}:\mathbb{R}^{p+1} \rightarrow \mathbb{R}^N$ is a neural network trained over a given training set made by a collection of input parameters $(\muv_i,t_i)_{i=1}^{N_{data}}$ and the corresponding snapshot matrix $\mathbf{U} \in \mathbb{R}^{N_h \times N_{data}}$, while $\mathbf{V} \in \mathbb{R}^{N_h \times N}$ is the POD projection matrix. \nicola{Then, }%Thus,
      under the Assumptions \ref{assumption:sampling} and \ref{assumption:solution-map}, we \nicola{have} %can bound the relative error by
    \begin{equation}
    \label{eq:error_decomposition}
        \mathcal{E}_R \le \mathcal{E}_S + \mathcal{E}_{POD} + \mathcal{E}_{NN},
    \end{equation}
    where:
    \begin{itemize}
        \item[\nicola{$\bullet$}] $\mathcal{E}_S = \mathcal{E}_S(\mathcal{G},\{(\muv_i,t_i)_{i=1}^{N_{data}}\}, N)$ is the sampling error, that satisfies $\mathcal{E}_S \xrightarrow{a.s.} 0$ as $N_s, N_t \rightarrow \infty$ and $\mathbb{E}[\mathcal{E}_S] = O(N_s^{-1/4} + N_t^{1/2})$%.
        \nicola{;\vspace{0.5em}}
        \item[\nicola{$\bullet$}] $\mathcal{E}_{POD} = \mathcal{E}_{POD}(\mathcal{G},\{(\muv_i,t_i)_{i=1}^{N_{data}}\}, N)$ is the POD projection error, that satisfies $\mathcal{E}_{POD} \xrightarrow{a.s.} \mathcal{E}_{POD,\infty}$ as $N_s, N_t \rightarrow \infty$, where $\mathcal{E}_{POD,\infty} = \mathcal{E}_{POD,\infty}(\mathcal{G}, N)$ is independent of the sampling criterion; \nicola{\vspace{0.5em}}
        \item[\nicola{$\bullet$}] $\mathcal{E}_{NN} = \mathcal{E}_{NN}(\mathcal{G},N,\hat{\bm{q}})$ is the approximation error of the neural network, which is arbitrarily low depending of the approximation capabilities of the network $\hat{\bm{q}}$.
    \end{itemize}
\end{theorem}

\begin{proof}
By means of the triangular inequality%(w.r.t. the $\|\cdot\|_{L^2_w}$ norm)
, we obtain
%\begin{equation*}
 %   \begin{aligned}
 %       \mathcal{E}_R &= \|\mathbf{u}(\muv, t)- \mathbf{V} \hat{\bm{q}}(\muv, t)\|_{L^2_w} \\ &=\biggl( \int_{\mathcal{P} \times \mathcal{T}} \frac{\|\mathbf{u}(\muv, t)- \mathbf{V} \hat{\bm{q}}(\muv, t)\|^2}{\|\mathbf{u}(\muv,t)\|^2} d(\muv,t) \biggr)^{1/2} \\
%& \le \biggl( \int_{\mathcal{P} \times \mathcal{T}} \frac{\|\mathbf{u}(\muv, t)  - \mathbf{V}\mathbf{V}^T\mathbf{u}(\muv, t)\|^2}{\|\mathbf{u}(\muv,t)\|^2} d(\muv,t) \biggr)^{1/2} +  \\ & \hspace{4cm} + \biggl(\int_{\mathcal{P} \times \mathcal{T}} \frac{\|\mathbf{V}\mathbf{V}^T\mathbf{u}(\muv, t) - \mathbf{V}\hat{\bm{q}}(\muv, t)\|^2}{\|\mathbf{u}(\muv,t)\|^2} d(\muv,t)\biggr)^{1/2}
%    \end{aligned}
%\end{equation*}
\nicola{\begin{multline}
    \label{eq:diseq}
       \mathcal{E}_R  =\left( \int_{\mathcal{P} \times \mathcal{T}} \frac{\|\mathbf{u}(\muv, t)- \mathbf{V} \hat{\bm{q}}(\muv, t)\|^2}{\|\mathbf{u}(\muv,t)\|^2} d(\muv,t) \right)^{1/2} = \|\mathbf{u}(\muv, t)- \mathbf{V} \hat{\bm{q}}(\muv, t)\|_{L^2_w} \\\vspace{0.5em}
 \le \|\mathbf{u}(\muv, t)  - \mathbf{V}\mathbf{V}^T\mathbf{u}(\muv, t)\|_{L^2_w} +  \|\mathbf{V}\mathbf{V}^T\mathbf{u}(\muv, t) - \mathbf{V}\hat{\bm{q}}(\muv, t)\|_{L^2_w}.
\end{multline}
According to the notation of Section~\ref{sec:three}, let $\bm{q}(\muv,t) := \mathbf{V}^T \mathbf{u}(\muv, t)$. We define}
%recalling that $\bm{q}(\muv,t) = \mathbf{V}^T \mathbf{u}(\muv, t)$ by definition, we define:
\begin{equation*}
      \mathcal{E}_{NN} := \biggl( \int_{\mathcal{P} \times \mathcal{T}} \frac{\|\mathbf{V}\bm{q}(\muv, t)- \mathbf{V}\hat{\bm{q}}(\muv, t)\|^2}{\|\mathbf{u}(\muv,t)\|^2} d(\muv,t) \biggr)^{1/2}
\end{equation*}
and notice that $\mathcal{E}_{NN}$ is the only error component that depends on the neural network approximation. Moreover, we can bound the remaining %part 
\nicola{term in \eqref{eq:diseq}} as
\nicola{\begin{equation*}
    %\begin{aligned}
        %&\biggl(\int_{\mathcal{P} \times \mathcal{T}} \frac{\|\mathbf{u}(\muv, t)- \mathbf{V}\mathbf{V}^T\mathbf{u}(\muv, t)\|}{\|\mathbf{u}(\muv,t)\|} d(\muv,t) \biggl)^{1/2}\\ &
        \|\mathbf{u}(\muv, t)- \mathbf{V}\mathbf{V}^T\mathbf{u}(\muv, t)\|_{L^{2}_{w}}\le m^{-1}
         \biggl(\int_{\mathcal{P} \times \mathcal{T}} \|\mathbf{u}(\muv, t)- \mathbf{V}\mathbf{V}^T\mathbf{u}(\muv, t)\|^2 d(\muv,t)\biggr)^{1/2}.
    %\end{aligned}
\end{equation*}
Let now} $\mathbf{K} = |\mathcal{P} \times \mathcal{T}| N_{data}^{-1} \mathbf{U}\mathbf{U}^T \in \mathbb{R}^{N_h \times N_h}$ %as 
\nicola{be} the \nicola{discrete} correlation matrix and %we call
\nicola{let} $\sigma_k^2$ \nicola{be} its eigenvalues. By employing the triangular inequality and the trivial inequality $\sqrt{a+b} \le \sqrt{a} + \sqrt{b}$ for $a,b \ge 0$,
\begin{equation*}
    \begin{aligned}
        & m^{-1}
         \biggl(\int_{\mathcal{P} \times \mathcal{T}} \|\mathbf{u}(\muv, t)- \mathbf{V}\mathbf{V}^T\mathbf{u}(\muv, t)\|^2 d(\muv,t)\biggr)^{1/2} \le 
         \\ \le &m^{-1}
         \biggl(\int_{\mathcal{P} \times \mathcal{T}} \|\mathbf{u}(\muv, t)- \mathbf{V}\mathbf{V}^T\mathbf{u}(\muv, t)\|^2 d(\muv,t) - \sum_{k>N} \sigma_k^2 + \sum_{k>N} \sigma_k^2\biggr)^{1/2} \le \\
         \le &m^{-1}
         \biggl(\biggl|\int_{\mathcal{P} \times \mathcal{T}} \|\mathbf{u}(\muv, t)- \mathbf{V}\mathbf{V}^T\mathbf{u}(\muv, t)\|^2 d(\muv,t) -  \sum_{k>N} \sigma_k^2\biggr| + \sum_{k>N} \sigma_k^2\biggr)^{1/2} \le
         \\ \le & m^{-1}
         \biggl| \int_{\mathcal{P} \times \mathcal{T}} \|\mathbf{u}(\muv, t)- \mathbf{V}\mathbf{V}^T\mathbf{u}(\muv, t)\|^2 d(\muv,t) -   \sum_{k>N} \sigma_k^2 \biggr|^{1/2} +  m^{-1} \sqrt{\sum_{k>N} \sigma_k^2}.
    \end{aligned}
\end{equation*}
\nicola{In light of this,} we define the sampling error as
\begin{equation*}
        \mathcal{E}_{S} := m^{-1}
         \biggl| \int_{\mathcal{P} \times \mathcal{T}} \|\mathbf{u}(\muv, t)- \mathbf{V}\mathbf{V}^T\mathbf{u}(\muv, t)\|^2 d(\muv,t) - \sum_{k>N} \sigma_k^2 \biggr|^{1/2},
\end{equation*}
and the POD error as
\begin{equation*}
    \mathcal{E}_{POD} := m^{-1} \sqrt{\sum_{k>N} \sigma_k^2}.
\end{equation*}
Thus, we obtain the inequality in \eqref{eq:error_decomposition}
\begin{equation*}
    \mathcal{E}_R \le \mathcal{E}_S + \mathcal{E}_{POD} + \mathcal{E}_{NN}.
\end{equation*}
In the last part of the proof we aim at showing the characteristic properties of $\mathcal{E}_S$ and $\mathcal{E}_{POD}$; recalling that
\begin{equation*}
    \sum_{k>N} \sigma_k^2 = \frac{|\mathcal{P} \times \mathcal{T}|}{N_{data}} \sum_{j=1}^{N_{data}}\|\mathbf{u}_j - \mathbf{V}\mathbf{V}^T\mathbf{u}_j\|^2,
\end{equation*}
we can write the sampling error in a slightly different form
\begin{equation*}
    \begin{aligned}
        \mathcal{E}_S = m^{-1}
         \biggl| \int_{\mathcal{P} \times \mathcal{T}} \|\mathbf{u}(\muv, t)- \mathbf{V}\mathbf{V}^T&\mathbf{u}(\muv, t)\|^2 d(\muv,t) - \\ & \frac{|\mathcal{P} \times \mathcal{T}|}{N_{data}} \sum_{j=1}^{N_{data}}\|\mathbf{u}_j - \mathbf{V}\mathbf{V}^T\mathbf{u}_j\|^2 \biggr|^{1/2}.
    \end{aligned}
\end{equation*}
Moreover, thanks to the compactness hypothesis of Assumption \ref{assumption:sampling} and the boundedness hypothesis of Assumption \ref{assumption:solution-map} we have that 
\begin{equation*}
    f(\muv,t) = \|\mathbf{u}(\muv, t)- \mathbf{V}\mathbf{V}^T \mathbf{u}(\muv, t)\|^2 \le M^2 \|\mathbf{I} - \mathbf{V}\mathbf{V}^T \|^2 < +\infty,
\end{equation*}
so that $f \in L^2(\mathcal{P} \times \mathcal{T})$. Thus, by means of Proposition \ref{prop:monte-carlo}, we conclude that $\mathbb{E}[\mathcal{E}_S] = O(N_s^{-1/4} + N_t^{-1/2})$.

Finally, since $\mathcal{E}_S$ and $\mathcal{E}_{POD}$ depend on the number of samples and snapshots in the training set, it is natural to verify their behavior in the \textit{infinite data} limit. Thanks to Assumption \ref{assumption:sampling}, by the Strong Law of Large Numbers, it is evident that $\mathcal{E}_S \xrightarrow{a.s.} 0$ as $N_s, N_t \rightarrow \infty$ and, by means of the results in Section~\ref{sec:three},
\begin{equation*}
\mathcal{E}_{POD} \xrightarrow{a.s.} \mathcal{E}_{POD,\infty} := m^{-1} \sqrt{\sum_{k>N} \sigma_{k,\infty}^2}, \qquad N_s, N_t \rightarrow \infty. \qedhere
\end{equation*}
%\nicola{$$\eqno\qed$$}
\end{proof}

\begin{remark}
    %We can actually improve 
    \nicola{The convergence rate for $\mathcal{E}_S$ can be improved} by modifying Assumption \ref{assumption:sampling}. Indeed, Monte Carlo sampling could be replaced by other strategies: for instance,  using Quasi-Monte Carlo techniques \cite{niederreiter1992random,caflisch_1998}, and under  suitable regularity  assumptions, one has $\mathbb{E}[\mathcal{E}_{S}] = O((\log(N_{s}))^{\frac{p+1}{2}}N_{s}^{-1/2} + N_t^{-1/2})$.
\end{remark}

\subsection{Lower bound for the relative error}
POD-DL-ROMs couple classical projection-based methods such as the POD with Deep Learning-based techniques that allow to correctly reproduce the nonlinearity of the parameter-to-POD-coefficient map $\mathcal{Q}$. This means that we still need to rely on the linear transformation represented by the POD matrix $\mathbf{V} \in \mathbb{R}^{N_h \times N}$ (or $\mathbf{V}_{\infty}$ in the \textit{infinite data} limit) to expand the neural network approximation of the POD coefficients. 

This last consideration is crucial: indeed, the fact that the POD-DL-ROM technique hinges upon a linear decomposition forces the relative error to still depend on the eigenvalues decay of the correlation matrix; the mentioned dependence is highlighted in the \textit{lower bound} result provided Theorem \ref{thm:lower_bound}.

First, we derive a lower bound for $\mathcal{E}_S + \mathcal{E}_{POD}$: we immediately prove that
\begin{equation*}
    \begin{aligned}
        & \mathcal{E}_S + \mathcal{E}_{POD} \ge
        \\ &\ge m^{-1} \biggl(\int_{\mathcal{P} \times \mathcal{T}} \|\mathbf{u}(\muv, t) - \mathbf{V}\mathbf{V}^T\mathbf{u}(\muv, t)\|^2   d(\muv,t)\biggr)^{1/2} \ge
        \\ &\ge m^{-1} \biggl(\min_{\mathbf{W} \in \mathbb{R}^{N_h \times N}: \mathbf{W}^T\mathbf{W} = \bm{I}} \int_{\mathcal{P} \times \mathcal{T}} \|\mathbf{u}(\muv, t) - \mathbf{W}\mathbf{W}^T\mathbf{u}(\muv, t)\|^2  d(\muv,t) \biggr)^{1/2} =
        \\ & =  m^{-1} \biggl(\int_{\mathcal{P} \times \mathcal{T}} \|\mathbf{u}(\muv, t) - \mathbf{V}_{\infty}\mathbf{V}_{\infty}^T\mathbf{u}(\muv, t)\|^2  d(\muv,t) \biggr)^{1/2} =
        \\ & = m^{-1} \sqrt{\sum_{k>N} \sigma_{k,\infty}^2} = \mathcal{E}_{POD,\infty},
    \end{aligned}
\end{equation*}
by trivially employing the definition of $\mathcal{E}_S$, $\mathcal{E}_{POD}$, and
$\mathbf{V}_{\infty}$. It is worth to remark that $\mathcal{E}_{POD,\infty}$ only depends on the eigenstructure of the continuous correlation matrix $\mathbf{K}_{\infty}$, while it is independent of the data sampling. 
Thus, in the following, we aim at showing that, up to a constant, $\mathcal{E}_{POD,\infty}$ represent a lower bound also for the relative error $\mathcal{E}_R$.
 
\begin{theorem}
    \label{thm:lower_bound}
    Under the same assumptions of Theorem \ref{thm:error_decomposition}, we have that 
    \begin{equation*}
        \mathcal{E}_R \ge \frac{m}{M} \mathcal{E}_{POD,\infty}.
    \end{equation*}
\end{theorem}

\begin{proof}
    We immediately notice that, by optimality of projection coefficients,
    \begin{equation*}
        \label{eq:lower_bound_overall}
        \begin{aligned}
         \mathcal{E}_R &= \int_{\mathcal{P} \times \mathcal{T}} \frac{\|\mathbf{u}(\muv, t)- \mathbf{V}\hat{\bm{q}}(\muv, t)\|}{\|\mathbf{u}(\muv, t)\|} d(\muv,t) \ge \int_{\mathcal{P} \times \mathcal{T}} \frac{\|\mathbf{u}(\muv, t) - \mathbf{V}\mathbf{V}^T\mathbf{u}(\muv, t)\|}{\|\mathbf{u}(\muv, t)\|} d(\muv,t),
         \end{aligned}
    \end{equation*}
    where we recall that $\mathbf{V}$ is the POD matrix computed via SVD using the discrete formulation and $ \mathbf{V}_{\infty}$ is relative to the continuous formulation.
    Then,
    \begin{equation*}
        \label{eq:lower_bound_first_term}
        \begin{aligned}
        (\mathcal{E}_{POD,\infty})^2 & =m^{-2} \sum_{k>N} \sigma_{k,\infty}^2 \\ &= m^{-2}\int_{\mathcal{P} \times \mathcal{T}}\|\mathbf{u}(\muv, t) -\mathbf{V}_{\infty}\mathbf{V}_{\infty}^T\mathbf{u}(\muv, t)\| d(\muv,t) \\ 
        & \le m^{-2}\int_{\mathcal{P} \times \mathcal{T}}\|\mathbf{u}(\muv, t) -\mathbf{V}\mathbf{V}^T\mathbf{u}(\muv, t)\|^2 d(\muv,t) \\
        & = m^{-2}\int_{\mathcal{P} \times \mathcal{T}} \frac{\|\mathbf{u}(\muv, t) -\mathbf{V}\mathbf{V}^T\mathbf{u}(\muv, t)\|^2}{\|\mathbf{u}(\muv, t)\|^2} \|\mathbf{u}(\muv, t)\|^2 d(\muv,t) \\
        & \le \frac{M^2}{m^2} \int_{\mathcal{P} \times \mathcal{T}} \frac{\|\mathbf{u}(\muv, t) -\mathbf{V}\mathbf{V}^T\mathbf{u}(\muv, t)\|^2}{\|\mathbf{u}(\muv, t)\|^2} d(\muv,t) \\
        & \le \frac{M^2}{m^2} (\mathcal{E}_R)^2,
        \end{aligned}
    \end{equation*}
    from which the thesis follows.
    %$$\eqno\qed$$
\end{proof}

\begin{remark}
    Since $\mathbf{V}_{\infty}$ is not available in practice, we cannot compute exactly $\mathcal{E}_{POD,\infty}$. In practice we can use a stricter bound: leveraging on quantities emerging from the proof, we actually employ
    \begin{equation*}
        \tilde{\mathcal{E}}_{POD} := m^{-1} \biggl( \int_{\mathcal{P} \times \mathcal{T}}\|\mathbf{u}(\muv, t) -\mathbf{V}\mathbf{V}^T\mathbf{u}(\muv, t)\|^2 d(\muv,t) \biggr)^{1/2}
    \end{equation*}
    when we either compute analytically (if possible) or estimate via Monte-Carlo.
\end{remark}

This result states that no matter how accurate the neural networks approximation is the relative error $\mathcal{E}_R$ is still bounded from below by the variance that is not explained by the POD projection. Additionally, the lower bound does not depend on how much data we gather for the supervised training phase. Of note, this is in agreement with the results provided in the analysis of other linear decomposition-based techniques, such as DeepONets \cite{deeponets_bounds_2021}.

\subsection{Upper bound for the relative error}
On the basis of the error decomposition and the perfect embedding hypothesis, we aim at providing the main result of this work, which is contained in the Theorem \ref{thm:upper_bound} and is endowed with a \textit{constructive} proof founded on the approximation results of \cite{yarotsky2017}. We remark that the present result is only valid for POD-DL-ROMs.

\begin{theorem}
    \label{thm:upper_bound}
    Let $\mathcal{G}: (\muv,t) \mapsto \mathbf{u}(\muv,t)$ for any $(\muv,t) \in \mathcal{P} \times \mathcal{T}$ be the parameter-to-solution map and suppose valid Assumptions \ref{assumption:sampling} and \ref{assumption:solution-map}. Let $\delta > 0$ and $ 0 < \varepsilon < 1$; suppose to have collected $N_{data} = N_{data}(\delta,\varepsilon)$ data samples into the snapshot matrix $\mathbf{U} \in \mathbb{R}^{N_h \times N_{data}}$. 
    Consider the $(\mathcal{P} \times \mathcal{T})$-discrete correlation matrix $\mathbf{K} = |\mathcal{P} \times \mathcal{T}| N_{data}^{-1}\mathbf{U}\mathbf{U}^T \in \mathbb{R}^{N_h \times N_h}$ and let $\sigma_{k}^2$ be its eigenvalues.
    Moreover, choose
    \begin{equation*}
        N = \arg\min \biggl\{j \in \mathbb{N}: \sum_{k>j} \sigma_{k}^2 \le \frac{m^2}{9} \varepsilon^2 \biggl\}.
    \end{equation*}
    We define the parameter-to-POD-coefficients map $\mathcal{Q}: (\muv,t) \mapsto \bm{q}(\muv,t)$ as $\mathcal{Q}(\muv,t) = \mathbf{V}^T \mathcal{G}(\muv,t)$ for any $ (\muv,t) \in \mathcal{P} \times \mathcal{T}$, where $\mathbf{V} \in \mathbb{R}^{N_h \times N}$ is the reduced rank-$N$ POD matrix computed via SVD.
    We assume that there exists $n>0$, $\Psi_*: \mathbb{R}^n \rightarrow \mathbb{R}^N, \Psi_*': \mathbb{R}^N \rightarrow \mathbb{R}^n$ that are respectively $s$-times and $s'$-times differentiable (with $s \gg s' \ge 2$), such that they enjoy the perfect embedding assumption stated in Definition  \ref{def:perfect_embedding}, namely
    \begin{equation*}
        \Psi_*(\Psi'_*(\bm{q}(\muv,t)) = \bm{q}(\muv,t) \qquad \forall (\muv,t) \in \mathcal{P} \times \mathcal{T}.
    \end{equation*}
    We let
    \begin{equation*}
        C_1 = \sup_{|\bm{\alpha}| \le s'} \sup_{\mathbf{v} \in \mathbb{R}^N} |D^{\bm{\alpha}} \Psi_*'(\mathbf{v})| \qquad C_2 = \sup_{|\bm{\alpha}| \le s} \sup_{\mathbf{w} \in \mathbb{R}^n} |D^{\bm{\alpha}} \Psi_*(\mathbf{w})|.
    \end{equation*}
    Then, there exists a constant $c = c(\mathcal{P},\mathcal{T},L,C_1,C_2,p,n,s,s')$ and a POD-DL-ROM architecture $\mathbf{V} \hat{\bm{q}} = \mathbf{V} \psi \circ \phi : \mathbb{R}^{p + 1} \rightarrow \mathbb{R}^{N}$ composed of a decoder $\psi : \mathbb{R}^{n} \rightarrow \mathbb{R}^{N}$ having at most:
    \begin{itemize}
        \item $L_{n \rightarrow N} = c\log(\varepsilon^{-1})$ layers,
        \item $w_{n \rightarrow N} = c N \varepsilon^{-n/(s-1)} \log(\varepsilon^{-1})$ active weights, 
    \end{itemize}
    and a reduced map $\phi: \mathbb{R}^{p + 1} \rightarrow \mathbb{R}^{n}$ having at most:
    \begin{itemize}
        \item $L_{(p+1) \rightarrow n} = c\log(\varepsilon^{-1})$ layers,
        \item $w_{(p+1) \rightarrow n} = c n \varepsilon^{-(p+1)} \log(\varepsilon^{-1})$ active weights,
    \end{itemize}
    such that $\mathbb{P}\{\mathcal{E}_R < \varepsilon\} > 1 - \delta$.
\end{theorem}

\begin{proof}
    We immediately notice that, choosing $N$ as in the theorem statement, we derive
    \begin{equation*}
        \label{eq:bound_pod_error}
        \mathcal{E}_{POD} = m^{-1} \sqrt{\sum_{k>N} \sigma_{k}^2} \le \frac{\varepsilon}{3}.
    \end{equation*}
    Then, we aim at bounding $\mathcal{E}_{S} = \mathcal{E}_{S}(N_s, N_t)$; under the Assumption \ref{assumption:sampling}, by the Weak Law of Large Numbers \cite{jacod_protter} we can infer the following statement:
    \begin{equation*}
        \forall \delta > 0, \quad \forall 0 < \varepsilon < 1, \quad \exists N_s, N_t:  \mathbb{P}\{\mathcal{E}_{S}(N_s, N_t) < \varepsilon/3\} > 1 - \delta.
    \end{equation*}
    Then, we are left to bound $\mathcal{E}_{NN}$:
    by means of the Cauchy-Schwarz and the H\"older inequalities, considering that $\|\mathbf{V}\|^2 = 1$, it is trivial that
    \begin{equation}
    \label{eq:upper_bound_E_NN}
        \begin{aligned}
         \mathcal{E}_{NN} &= \biggl(\int_{\mathcal{P} \times \mathcal{T}}\frac{\|\mathbf{V}\bm{q}(\muv,t) - \mathbf{V}\hat{\bm{q}}(\muv, t)\|^2}{\|\mathbf{u}(\muv, t)\|^2} d(\muv,t)\biggr)^{1/2} \\ &\le m^{-1} \biggl(\int_{\mathcal{P} \times \mathcal{T}}\|\mathbf{V}\bm{q}(\muv,t) - \mathbf{V}\hat{\bm{q}}(\muv, t)\|^2 d(\muv,t)\biggr)^{1/2} \\ &\le m^{-1} \biggl(|\mathcal{P} \times \mathcal{T}| \sup_{(\muv,t) \in \mathcal{P} \times \mathcal{T}} \|\mathbf{V}\bm{q}(\muv,t) - \mathbf{V}\hat{\bm{q}}(\muv, t)\|^2 \biggr)^{1/2} 
        \\ &\le m^{-1} \biggl(|\mathcal{P} \times \mathcal{T}| \|\mathbf{V}\|^2\sup_{(\muv,t) \in \mathcal{P} \times \mathcal{T}} \|\bm{q}(\muv,t) - \hat{\bm{q}}(\muv, t)\|^2 \biggr)^{1/2} 
        \\ &= m^{-1} |\mathcal{P} \times \mathcal{T}|^{1/2} \sup_{(\muv,t) \in \mathcal{P} \times \mathcal{T}} \|\bm{q}(\muv,t) - \hat{\bm{q}}(\muv, t)\| ,
        \end{aligned}
    \end{equation}
    Therefore, we are left to bound the error due to the neural network approximation of the map $\mathcal{Q}$, namely
    \begin{equation*}
        \sup_{(\muv,t) \in \mathcal{P} \times \mathcal{T}} \|\bm{q}(\muv,t) - \hat{\bm{q}}(\muv, t)\|.
    \end{equation*}
    Firsly, we notice that we can take $n \le 2p + 3$, since $\mathcal{G}$ (and consequently $\mathcal{Q}$) is Lipschitz-continuous (see \textit{Theorem 3} in \cite{dlrom_bounds}). Then, we proceed as in \cite{dlrom_bounds}, by employing two different steps:
    \begin{itemize}
        \item Consider the reduced manifold $\mathcal{S}_N := \{\bm{q} = \mathcal{Q}(\muv, t): (\muv,t) \in \mathcal{P} \times \mathcal{T}\}$; then $\mathcal{V}_n = \Psi'_*(S_N)$ is such that $\diam(\mathcal{V}_N) \le LC_1 \diam(\mathcal{P} \times \mathcal{T})$, thanks to the Lipschitz-continuity hypothesis provided by Assumption \ref{assumption:solution-map}. Thus, by Theorem due to Gühring et al. \cite{guhring_relu_2019} recalled in Section~\ref{sec:1_2}, there exists a ReLU DNN $\psi: \mathbb{R}^n \rightarrow \mathbb{R}^N$ such that
        \begin{equation}
            \label{eq:bound_decoder}
            \begin{aligned}
            &\sup_{\mathbf{v} \in \mathcal{V}_n} \|\psi(\mathbf{v}) - \Psi_*(\mathbf{v})\| < \frac{m}{6}|\mathcal{P} \times \mathcal{T}|^{-1/2} \varepsilon
            \\
            &\esssup_{\mathbf{v}, \mathbf{v}' \in \mathcal{V}_n} \frac{|(\psi - \Psi_*)(\mathbf{v}) - (\psi - \Psi_*)(\mathbf{v}')|}{|\mathbf{v} - \mathbf{v}'|} < \frac{m}{6}|\mathcal{P} \times \mathcal{T}|^{-1/2} \varepsilon,
            \end{aligned}
        \end{equation}
        with $L_{n \rightarrow N} = c\log(\varepsilon^{-1})$ layers and $w_{n \rightarrow N} = c N \varepsilon^{-n/(s-1)} \log(\varepsilon^{-1})$ active weights. Notice that the Lipschitz constant of $\psi$ is bounded by the quantity $C_3 = C_2 + \frac{m}{6} |\mathcal{P} \times \mathcal{T}|^{-1/2}$;
        \item Setting $\phi_*(\muv,t) = \Psi'_*(\bm{q}(\muv,t)) \quad \forall (\muv,t) \in \mathcal{P} \times \mathcal{T}$, we notice that it is Lipschitz-continuous, with constant bounded by $LC_1$, and thus, by the Theorem due to Yarotski \cite{yarotsky2017} recalled in Section~\ref{sec:1_1}, there exists a ReLU DNN $\phi: \mathbb{R}^{p+1} \rightarrow \mathbb{R}^n$ such that
        \begin{equation}
            \label{eq:bound_reduced}
            \sup_{(\muv,t) \in \mathcal{P} \times \mathcal{T}} \|\phi(\muv,t) - \phi_*(\muv,t)\| < \frac{m}{6 C_3}|\mathcal{P} \times \mathcal{T}|^{-1/2} \varepsilon,
        \end{equation}
        with $L_{(p+1) \rightarrow n} = c\log(\varepsilon^{-1})$ layers and $w_{(p+1) \rightarrow n} = c n \varepsilon^{-(p+1)} \log(\varepsilon^{-1})$ active weights.
    \end{itemize}
    Moreover, let $\hat{\bm{q}} = \psi \circ \phi : \mathbb{R}^{p + 1} \rightarrow \mathbb{R}^{N}$ be the underlying neural network of the POD-DL-ROM.
    Then, by means of the triangular inequality, the perfect embedding Assumption, the definition of $\phi_*$, and the Lipschitz-continuity of $\psi$, we derive:
    \begin{equation*}
        \label{eq:bound_network_error_2}
        \begin{aligned}
        &\sup_{(\muv,t) \in \mathcal{P} \times \mathcal{T}} \|\bm{q}(\muv,t) - \hat{\bm{q}}(\muv, t)\| \\
        \le &\sup_{(\muv,t) \in \mathcal{P} \times \mathcal{T}} (\|\Psi_*(\Psi_*'(\bm{q}(\muv,t)) - \psi(\phi_*(\muv,t))\| + \|\psi(\phi_*(\muv,t)) - \psi(\phi(\muv,t))\|) \\
        \le & \sup_{\mathbf{v} \in \mathcal{V}_N} \|\psi(\mathbf{v}) - \Psi_*(\mathbf{v})\| + C_3 \sup_{(\muv,t) \in \mathcal{P} \times \mathcal{T}} \|\phi(\muv,t) - \phi_*(\muv,t)\| < \frac{m}{3}|\mathcal{P} \times \mathcal{T}|^{-1/2} \varepsilon,
        \end{aligned}
    \end{equation*}
    employing the bounds \eqref{eq:bound_decoder} and \eqref{eq:bound_reduced}.
    Then, plugging the last inequality in \eqref{eq:upper_bound_E_NN} we can state that $\mathcal{E}_{NN} < \frac{\varepsilon}{3}$.
    Finally, by means of the error decomposition formula, we derive the desired bound
    \begin{equation*}
        \mathcal{E}_R \le \mathcal{E}_{POD} + \mathcal{E}_S + \mathcal{E}_{NN} < \frac{\varepsilon}{3} + \frac{\varepsilon}{3} + \frac{\varepsilon}{3} = \varepsilon,
    \end{equation*}
    with probability greater than $1 - \delta$.
\end{proof}

\begin{remark}
The DL-ROM paradigm proposed in \cite{dlrom_2021} and applied to cardiac electrophysiology in \cite{dlrom_electrophys}, has been  theoretically analyzed in \cite{dlrom_bounds}, providing approximation bounds and a complexity analysis, which shows that in general DL-ROMs suffer from \textit{curse of dimensionality} with respect the number of high-fidelity dofs $N_h$. Relying on the present Theorem \ref{thm:upper_bound}, we demonstrate how the preliminary dimensionality reduction through POD affects both the complexity of the POD-DL-ROM and its approximation capabilities. 
Indeed, POD-DL-ROMs avoid the curse of dimensionality of the DL-ROMs at the cost of discarding the \textit{small scales} contribution, which might be however relevant when considering, e.g., highly nonlinear problems showing a slow eigenvalue decay. On the other hand, POD-DL-ROMs provide a neural network architecture with a lower number of \textit{trainable} weights, thus yielding a lighter training procedure in practice. Finally, we can highlight that the \textit{a priori} choice of employing DL-ROMs or POD-DL-ROMs must be based exclusively on the linear reducibility of the problem and the availability of computational resources.
\end{remark}

\section{Comparative analysis with deep learning-based existing strategies}
\label{sec:five}

On the basis of the results of the previous section, we comment the advantages of POD-DL-ROMs when compared with other deep learning-based existing strategies present in the literature, namely:
\begin{itemize}
    \item simple DNNs to approximate the POD (or Kernel-POD) coefficients, that results in the widely used POD+DNN approach \cite{chen_2021,hesthaven_2018,salvador_2021,wang_2019}; \smallskip
    
    \item the POD-DeepONets architecture, which was proposed in \cite{deeponet_vs_fno_2022} and based on the classical DeepONets approach \cite{lu2021learning}; \smallskip
    
    \item the technique presented in \cite{oleary-2022}, which aims at reconstructing the parameter-to-solution map by coupling linear projection methods and residual networks and which we will hereon refer to as lin+ResNets; \smallskip
    
    \item the CNNs architecture for operator learning proposed in \cite{appx_convenet_2022}, whose analysis is based on the Fourier decomposition.
\end{itemize}

\subsection{POD-DL-ROMs vs POD+DNNs: a matter of regularity}
\label{subsec:POD_DNN}

The purpose of this subsection is to highlight how the POD-DL-ROM approach provides a suitable setting to establish tighter bounds on the model complexity when compared to generic POD+DNNs, especially when the parameter-to-solution map is not regular.

It is worth to remark that, under the hypothesis of Theorem \ref{thm:upper_bound}, the number of layers of the POD-DL-ROM network architecture is expected to scale as
\begin{equation*}
    L_{POD-DL-ROM} = O(\log(\varepsilon^{-1})),
\end{equation*}
while the total number of active weights behaves as
\begin{equation*}
    \begin{aligned}
        w_{POD-DL-ROM} &= w_{n \rightarrow N} + w_{(p+1) \rightarrow n} \\
        &= O(N \varepsilon^{-n/(s-1)} \log(\varepsilon^{-1})) + O(n \varepsilon^{-(p+1)} \log(\varepsilon^{-1})).
    \end{aligned}
\end{equation*} 

We expect that, in general  $n \ll N = N(\varepsilon)$; moreover, $n \le 2p + 3$ since the parameter-to-POD-coefficients map is Lipschitz-continuous, due to Assumption \ref{assumption:solution-map}. Thus, it is evident that the majority of the neural network complexity amounts to the decoder, which has to perform the most difficult task, namely, decoding the information provided by the latent coordinates. Instead, the reduced network only aims at providing an alternative representation of the time-parameters vector $(\muv,t)$ such that it makes as easy as possible for the decoder to reconstruct the POD coefficients. Noting that $w_{POD-DL-ROM}$ depends exponentially on $s$, we can control the complexity of the POD-DL-ROM by choosing $s$ as large as possible, namely, $s \gg s' \ge 2$.

Essentially, we aim at  finding a representation of POD coefficients of the form
\begin{equation}
\label{eq:perfect_embedding}
    \Psi_*(\Psi'_*(\bm{q}(\muv,t)) = \bm{q}(\muv,t) \qquad \forall (\muv,t) \in \mathcal{P} \times \mathcal{T},
\end{equation}
through the composition of an encoder $\Psi'_*$ that absorbs all the irregularity of the identity map $\mathcal{I} = \Psi_* \circ \Psi'_*$, and a decoder $\Psi_*$ that is extremely regular.
We highlight that the perfect embedding Assumption stated in Definition \ref{def:perfect_embedding} is critical; indeed, under the hypothesis of Theorem \ref{thm:upper_bound}, leaving out only the perfect embedding assumption, we may be tempted to trivially use Yarotski's Theorem \cite{yarotsky2017} to construct a ReLU DNN which has $L_{POD+DNN}$  layers and  $w_{POD+DNN}$ active weights, where
\begin{equation*}
    \begin{aligned}
        &L_{POD+DNN} = O(\log(\varepsilon^{-1})) \\
        &w_{POD+DNN} = O(N \varepsilon^{-(p+1)}\log(\varepsilon^{-1})),
    \end{aligned}
\end{equation*}
in order to control the relative error with $\mathcal{E}_R < \varepsilon$. Notice that:
\begin{itemize}
    \item the number of layers $L_{POD+DNN}$ is of the same order as $L_{POD-DL-ROM}$; 
    \item the estimate of the number of active weights  $w_{POD+DNN}$ can only take advantage of mild regularity assumptions on $\mathcal{G}$ (and $\mathcal{Q}$), that is only Lipschitz-continuous.
\end{itemize}

However, it is evident that  Theorem \ref{thm:upper_bound} only provides a theoretical result offering a different perspective in order to enhance the complexity estimate of POD+DNN. Indeed, within the framework stated by Theorem \ref{thm:upper_bound}, given an accuracy level $\varepsilon$ one could take advantage of the POD-DL-ROM theoretical setting, and thus the perfect embedding Assumption, to construct a proper architecture that approximates the parameter-to-solution map keeping $\mathcal{E}_R < \varepsilon$ -- and, then, notice that the resulting architecture is indeed in general a POD+DNN. The difference in practice is represented by the training procedure. Indeed, notice that training a network like the one involved in a POD+DNN with the classical supervised loss formulation, by letting $\omega_n = 0$ in \eqref{eq:POD-DL-ROM-loss} and thus without taking advantage of the encoder, does not ensure to recover an adequate representation in the latent space. Instead, if we train the network relying on thje POD-DL-ROM paradigm, namely taking $\omega_n > 0$ in \eqref{eq:POD-DL-ROM-loss}, we actually employ the encoder to implicitly enforce the architecture to satisfy the perfect embedding Assumption, and then discard the encoder in the online testing phase.

Suppose now that $N \gg n$: trivially, we have that $w_{DNN} \gtrsim w_{(p+1) \rightarrow n}$; moreover, $w_{POD+DNN} \gtrsim w_{n \rightarrow N}$, upon requiring that $\frac{n}{(s-1)} < p + 1$, that provides an estimate for the regularity of the decoder in the representation \eqref{eq:perfect_embedding}, that is $s > \frac{n}{p + 1} + 1$. In practice, given that $n \le 2p + 3$, we can safely assume that $s \gtrsim 3 + \frac{1}{p + 1} $ and finally $s \gtrsim 4$. Thus, if the parameter-to-solution map $\mathcal{G}$ is only Lipschitz-continuous, if the perfect embedding Assumption is satisfied for $s \ge 4$, POD-DL-ROMs achieve a tighter bound on the model complexity when compared to general POD+DNN approaches: this is due to the fact that there exists a better representation (in terms of regularity) $\phi_*(\muv,t)$ for the time-parameters vector $(\muv,t)$ that can be recovered by the reduced network.

Until now, we considered the case where the parameter-to-solution map is only Lipschitz-continuous; however, it is interesting to consider cases where we can verify that the map $\mathcal{G}$ shows higher regularity, and see how this increased regularity affects the complexity of both POD-DL-ROMs and POD+DNNs in terms of number of active weights. Indeed, by means of similar arguments employed previously, and thanks to the Theorem due to Yarotski \cite{yarotsky2017} recalled in Section~\ref{sec:1_1}, assuming that $\mathcal{G} \in W^{r,+\infty}(\mathcal{P} \times \mathcal{T}; \mathbb{R}^{N_h})$, we obtain that 
\begin{equation*}
    \begin{aligned}
        &w_{POD+DNN} = O(N \varepsilon^{-(p+1)/r}\log(\varepsilon^{-1})).
    \end{aligned}
\end{equation*}
Thanks to the fact that the \textit{exact} reduced map $\phi_*$ of Theorem \ref{thm:upper_bound} now would be $\min\{r,s'\}$-times differentiable,
\begin{equation*}
    \begin{aligned}
         w_{POD-DL-ROM} &= w_{n \rightarrow N} + w_{(p+1) \rightarrow n} \\
        &= O(N \varepsilon^{-n/(s-1)} \log(\varepsilon^{-1})) + O(n \varepsilon^{-(p+1)/\min\{r,s'\}} \log(\varepsilon^{-1}))
    \end{aligned}
\end{equation*}
Assuming that $N \gg n$, it is trivial to verify that $w_{POD+DNN} \gtrsim w_{(p+1) \rightarrow n}$; furthermore, it is valid that $w_{POD+DNN} \gtrsim w_{n \rightarrow N}$ if $\frac{n}{(s-1)} > \frac{p+1}{r}$, that is $s > \frac{nr}{p+1} + 1$, which gives the estimate $s \gtrsim (2 + \frac{1}{p+1})r + 1$ and finally $s \gtrsim 3r + 1$. The meaning of the last estimate is that, if the parameter-to-solution map is extremely regular (namely, $r \rightarrow \infty$), it becomes more and more difficult for the POD-DL-ROMs to guarantee lower complexity than simple POD+DNNs, since the perfect embedding Assumption should be verified for $s \rightarrow \infty$. This is rather intuitive: indeed, if the parameter-to-solution map is extremely regular, we do not need a to recover a better representation $\phi_*(\muv,t)$ for the time-parameter vector $(\muv,t)$ in order to make it easier for the underlying neural network to learn the solution manifold. 

\subsection{POD-DeepONets and POD-DL-ROMs: a comparison}
In this subsection, we aim at analyzing the POD-DeepONet architecture from a theoretical standpoint, showing the close relationship with POD-DL-ROMs when dealing with problems whose general formulation can be reduced to \eqref{eq:general_formulation_FOM}. 
We let $X$ be a Banach space and consider a compact subset $K_1 \subset X$ and a compact subset $K_2 \subset \mathbb{R}^d$, where $d$ denotes the number of spatial (or spatio-temporal) dimensions of the problem at hand. Defining $W \subset C(K_1)$ as a compact subset, we suppose that we aim at learning the operator $\mathcal{G}_{\infty \rightarrow \infty}: W \rightarrow C(K_2)$, where the subscript highlights that the considered operator is a map between infinite-dimensional spaces. We first consider a DeepONet architecture \cite{lu2021learning} employed to reconstruct $\mathcal{G}_{\infty \rightarrow \infty}$, which in its \textit{unstacked} formulation consists in the combination of the output of two different neural networks through the scalar product. In particular, we define the \textit{branch net} $\bm{b}: W \rightarrow \mathbb{R}^N$ as the neural network that processes information about the input function $\phi \in W$, and the \textit{trunk net} $\bm{\tau}: \mathbb{R}^d \rightarrow \mathbb{R}^N$, which aims at encoding the coordinate input $y \in \mathbb{R}^d$ in a set of basis functions.  Then, we can define the DeepONet approximation as
\begin{equation}
    \label{eq:deeponet_decomposition}
    \mathcal{G}_{\infty \rightarrow \infty}(\phi)(y) \approx \hat{G}(\phi)(y) = \bm{b}(\phi) \cdot \bm{\tau}(y).
\end{equation}
and note that $N$ describes the number of basis functions employed in the decomposition \eqref{eq:deeponet_decomposition}; thus, $N$ plays the same role as the POD dimension in the POD-DL-ROM architecture.
Based on the analysis proposed in \cite{deeponets_bounds_2021}, we can split the DeepONet operator $\hat{G}: W \rightarrow C(K_2)$ into $\hat{G} = \mathcal{R}_{\tau} \circ \mathcal{A}_{m \rightarrow N} \circ \mathcal{E}_m$, where $\mathcal{E}_m$, $\mathcal{A}_{m \rightarrow N}$ and $\mathcal{R}_{\tau}$ are defined as follows:
\begin{itemize}
    \item the \textit{encoder} operator is defined as the map $\mathcal{E}_m: C(K) \rightarrow \mathbb{R}^m$, such that, given $x_i \in K_1, \forall i=1,\ldots,m$:
    \begin{equation*}
        \mathcal{E}_m(\psi) = [\psi(x_1), \psi(x_2), ..., \psi(x_m)]^T \qquad \forall \psi \in C(K_1).
    \end{equation*}
    It is worth to notice that $\mathcal{E}_m$ is well defined since any continuous function can be evaluated pointwise; \smallskip
    
    \item $\mathcal{A}_{m \rightarrow N}: \mathbb{R}^m \rightarrow \mathbb{R}^N$ is the \textit{approximation} operator; thus, we can decompose the branch net of the DeepONet operator as $\bm{b} = \mathcal{A}_{m \rightarrow N} \circ \mathcal{E}_m$; \smallskip
    
    \item recalling that $\bm{\tau}: \mathbb{R}^d \rightarrow \mathbb{R}^N$ is the trunk net, we define the $\tau$-induced \textit{reconstructor} operator $\mathcal{R}_{\tau}: \mathbb{R}^N \rightarrow C(K_2)$ as
    \begin{equation*}
        \mathcal{R}_{\tau}(\bm{\xi}) = \bm{\xi} \cdot \bm{\tau} \qquad \forall \bm{\xi} \in \mathbb{R}^N.
    \end{equation*}
\end{itemize}
In a more compact formulation, we retrieve the classical architecture of the DeepONets, namely:
\begin{equation*}
    \hat{G}(\phi) = \mathcal{R}_{\tau} \circ (\mathcal{A}_{m \rightarrow N} \circ \mathcal{E}_m(\phi)) = \mathcal{R}_{\tau} \circ \bm{b}(\phi) = \bm{b}(\phi) \cdot \bm{\tau}.
\end{equation*}
 POD-DeepONets were recently introduced in \cite{deeponet_vs_fno_2022} and the test cases considered within the paper confirm better approximation accuracy when compared with classical DeepONets: the methodology consists in substituting the \textit{trunk net} with the corresponding row of the POD matrix. The drawback is that POD-DeepONets can only approximate operators defined as $\mathcal{G}_{\infty \rightarrow N_h}: W \rightarrow \mathbb{R}^{N_h}$, losing the capability of mapping between infinite-dimensional spaces. 
 
 Supposing to initially deal with stationary, time-independent problems and denoting by $\mathbf{v}_j \in \mathbb{R}^N$ the $j$-th row of the POD matrix $\mathbf{V} \in \mathbb{R}^{N_h \times N}$, we define the \textit{expansion} operator $L_{\mathbf{v}_j}: \mathbb{R}^N \rightarrow \mathbb{R}^{N_h}$ as
\begin{equation*}
   L_{\mathbf{v}_j}(\bm{\xi}) = \bm{\xi} \cdot \mathbf{v}_j \qquad \forall \bm{\xi} \in \mathbb{R}^N,
\end{equation*}
and the POD-DeepONet operator as
\begin{equation*}
    [\mathcal{G}_{\infty \rightarrow N_h}(\phi)]_j \approx [\hat{G}_{POD-DeepONet}(\phi)]_j = \hat{\bm{q}}(\phi) \cdot \mathbf{v}_j = L_{\mathbf{v}_j} \circ \mathcal{A}_{m \rightarrow N} \circ \mathcal{E}_m(\phi),
\end{equation*}
$\forall j=1,\ldots,N_h$, where $\hat{\bm{q}}$ is the corresponding \textit{branch net}, which now approximates the underlying POD coefficients. It is worth to notice that, by employing the vector formulation, we can write:
\begin{equation*}
    \mathcal{G}_{\infty \rightarrow N_h}(\phi) \approx \hat{G}_{POD-DeepONet}(\phi) = \mathbf{V} \hat{\bm{q}}(\phi).
\end{equation*} 
Then, we need to adapt the POD-DeepONet framework to the problem considered within this work \eqref{eq:general_formulation_FOM}, where even the input parameter space is finite-dimensional, thus eliminating the need of the \textit{encoder} operator $\mathcal{E}_m$. Indeed, POD-DeepONets for finite-dimensional-input problems involving the reconstruction of the map $\mathcal{G}_{p \rightarrow N_h}: \mathbb{R}^p \rightarrow \mathbb{R}^{N_h}$ take the form
\begin{equation*}
    [\mathcal{G}_{p \rightarrow N_h}(\muv)]_j \approx [\hat{G}_{POD-DeepONet}(\muv)]_j = \hat{\bm{q}}(\muv) \cdot \mathbf{v}_j = L_{\mathbf{v}_j} \circ \mathcal{A}_{p \rightarrow N},
\end{equation*}
$\forall j=1,\ldots,N_h$, or in a more compact way 
\begin{equation*}
    \mathcal{G}_{p \rightarrow N_h}(\muv) \approx \hat{G}_{POD-DeepONet}(\muv) = \mathbf{V} \hat{\bm{q}}(\muv),
\end{equation*}
where $\muv \in \mathcal{P} \subset \mathbb{R}^{p}$, $\mathcal{P}$ compact. It is worth to notice that in this case the \textit{branch net} coincides with the \textit{approximation} operator $\mathcal{A}_{m \rightarrow N}$.

Finally, in order to include also the time-dependence, we could adopt two different strategies:
\begin{itemize}
    \item we could treat the time $t$ as a spatial coordinate in a DeepONet-like way, leading to a POD matrix of dimension $N_hN_t \times N$, that however increases the possible impact of the \textit{curse of dimensionality}, however offering the opportunity to deal with time-dependent basis functions; 
    \item alternatively, we may consider the time $t$ as an additional parameter, a choice which reduces the computational requirements and is consistent with the POD-DL-ROM approach, leading to the construction of time-independent global spatial basis functions.
\end{itemize}

Within this comparison, for the sake of consistency, we choose to employ this latter approach. Thus, aiming at reconstructing the map $(\muv,t) \mapsto \mathbf{u}(\muv,t)$, we could employ different neural network architectures; for instance, if we choose to employ a DL-ROM architecture as the \textit{branch net} of the POD-DeepONets, we retrieve the POD-DL-ROM approach, while employing a \textit{vanilla} DNN as the \textit{branch net} results in the POD+DNN approach. The comparison between POD-DL-ROM and POD+DNNs is extensively treated in the previous subsection.

Finally, inspired by the DeepONet approach, we notice that extending the content of the present paper to the case of  infinite-dimensional input parameters is straightforward and introduces an additional source of error, namely the \textit{encoding} error, that ultimately depends on the variability of the input parameters and their spatial discretization; for a thorough discussion on the topic, we refer the reader to, e.g., \cite{deeponets_bounds_2021}.

\subsection{Learning POD coefficients with ResNets}
The ResNets-based approach proposed in \cite{oleary-2022} couples linear decompositions and residual networks (ResNets) to reconstruct field-to-solution maps, an approach which is inherently close to POD-DL-ROMs. In this case, we start our analysis of the technique by examining the proposed architecture, and by adapting it to the problem formulation considered within the present work. 

Indeed, we immediately notice that the lin+ResNet architecture needs that every residual layer has input dimension equal to the output dimension layer output dimension: iterating, for a fully residual network, we must require that the input of the network has the same dimension of the network output. Such a constraint in the architecture is managed in \cite{oleary-2022} by projecting both the input fields and the output targets onto two linear subspace of equal dimension $N \ll N_h$, where $N_h$ is the FOM dimension. Then, the output targets are numerically approximated on the same mesh and projected onto a subspace of dimension $N$, too. The approach results in the sequence of maps:
\begin{equation*}
    \mathbb{R}^{N_h} \xrightarrow{lin. proj.} \mathbb{R}^N \xrightarrow{residual} \mathbb{R}^N \xrightarrow{residual}... \xrightarrow{residual} \mathbb{R}^N \xrightarrow{lin. lift.} \mathbb{R}^{N_h},
\end{equation*}
where the linear projection is usually carried out by employing POD, Karhunen-Loève expansions \cite{kle_2006} or active subspaces \cite{active_subspaces_20}.
However, when dealing with finite dimensional parameter inputs instead of fields (for instance $(\muv,t) \in \mathbb{R}^{p+1}$ with $p + 1 < N$), it may occur that the ResNet input dimension ($p+1$) is different from the output dimension $N$; to fill the gap, it is necessary to employ for instance a dense layer $\mathbb{R}^{p+1} \rightarrow \mathbb{R}^{N}$ as the first layer of the architecture. Thus, we will consider the sequence of maps:
\begin{equation*}
    \mathbb{R}^{p} \xrightarrow{dense} \mathbb{R}^N \xrightarrow{residual} \mathbb{R}^N \xrightarrow{residual}... \xrightarrow{residual} \mathbb{R}^N \xrightarrow{lin. lift.} \mathbb{R}^{N_h}.
\end{equation*}

The lin+ResNets approach ultimately aims at providing a constructive way to build a neural network in terms of \textit{breadth} and \textit{depth}.

The \textit{breadth}, which may be intuitively defined as the maximum number of neurons per layer in the network, coincides with $N$, the characteristic dimension of the preliminary dimensionality reduction. In order to favour compressed representations, the authors of \cite{oleary-2022} suggest keeping as low as possible the \textit{latent} dimension $k$ of the ResNet, which can be identified with the dimension of the nonlinearity added at each layer. Indeed, the residual map between the layer $\mathbf{z}_{l} \in \mathbb{R}^N$ and $\mathbf{z}_{l+1} \in \mathbb{R}^N$ can be identified with
\begin{equation*}
    \mathbf{z}_{l+1} = \mathbf{z}_l + \mathbf{W}_{1l}\sigma(\mathbf{W}_{0l} \mathbf{z}_l + \mathbf{b}_l),
\end{equation*}
where $\mathbf{W}_{0l} \in \mathbb{R}^{N \times k}$, $\mathbf{W}_{1l} \in \mathbb{R}^{k \times N}$, $\mathbf{b}_l \in \mathbb{R}^k$ and $\sigma$ is the activation function; the total number of weights per layer is then $O(Nk)$.
However, in contrast to our approach, they did not propose a way to identify $k$: we remark that the discussion on the latent dimension $n$ of the POD-DL-ROM architecture is fundamental because it allows to set a tighter bound on the complexity of the decoder network in terms of active weights. 

Furthermore, the authors developed approximation bounds on the underlying ResNet complexity in terms of its \textit{depth}, employing the connection between ResNets, Neural ODE and control flows \cite{rtqchen_2018}. The bound on the ResNet \textit{depth} enable the user to control the $\ell^2$ error on the solution (and by extension the relative error too) with a suitable bound $\varepsilon$ by employing $O(\varepsilon^{-1})$ layers.
Thus, we can straightforwardly state that, on the basis of the complexity analysis, POD-DL-ROMs outperform the ResNets-based approach in terms of number of layers:
\begin{equation*}
    \begin{aligned}
    L_{lin+ResNets} = O(\varepsilon^{-1}) &\gtrsim O(\log(\varepsilon^{-1})) = L_{POD-DL-ROM} 
    \end{aligned}
\end{equation*}
and number of active weights:
\begin{equation*}
    \begin{aligned}
    w_{lin+ResNets} = O(Nk \varepsilon^{-1}) \gtrsim &\hspace{1mm} O(N \varepsilon^{-n/(s-1)} \log(\varepsilon^{-1})) + \\
    &+ O(n \varepsilon^{-(p+1)} \log(\varepsilon^{-1}))) = w_{POD-DL-ROM},
\end{aligned}
\end{equation*}
supposing for instance $N \gg n$ and $s \ge n + 1$, which are reasonable assumptions. Indeed, $N \gg n$ is satisfied when the nonlinear Kolmogorov $n$-width decays much faster that the eigenvalue decay of the correlation matrix, a phenomenon that is usually encountered in applications; the condition $s \ge n + 1$ is valid by ensuring $s \ge 2p + 4$, that is the decoder map must be sufficiently regular.

Despite the disadvantage on the complexity front, we remark that ResNets constitute one of the most suitable paradigms to implement adaptive-depth architectures, since adding a layer to an already trained architecture can produce an arbitrary small perturbation on the network output; for a more detailed analysis on the lin+ResNets training, we refer the reader to \cite{oleary-2022}.

\subsection{The effect of the POD basis optimality on the network complexity}

Within this subsection, our purpose is finally to show how choosing the POD basis as global spatial basis function in the linear decomposition leads to a reduced complexity of the underlying neural network, comparing in details CNNs for operator learning and POD-DL-ROMs. In particular, we notice that, within the POD-DL-ROM approach, the reconstruction of the approximated solution at the high-fidelity level depends on the decomposition assumption $\mathbf{u}(\muv,t) \approx \sum_{j < N} \hat{q}_j(\muv,t) \mathbf{v}_j$, where $N$ denotes the POD dimension. Analogously, the recent work on the approximation bounds for CNNs proposed in \cite{appx_convenet_2022} strives to reconstruct a decomposition between global spatial basis functions that are strictly related to the Fourier modes, and a set of coefficients, that is,  $\mathbf{u}(\muv,t) \approx \sum_{j<C} \hat{a}_j(\muv,t) \bm{f}_j$, where the sum is over $C$ terms (the number of channels in the input and output is $O(C)$). 

In the following, we assume that $u(\cdot,\muv,t) \in C^\alpha(\Omega)$ for any $ (\muv,t) \in \mathcal{P} \times \mathcal{T}$, being $\alpha \ge 1$ the spatial regularity, and $\varepsilon > 0$ is the desired accuracy level; we then describe the three main differences between the CNN-based approach and the POD-DL-ROM technique:
\begin{itemize}
    \item The convolutional block is limited to uniformly spaced mesh points ($h$ is the spacing parameter) in square domains, while POD-DL-ROMs are more versatile both in terms of the domain shape and the mesh properties.
    \item The architecture proposed in \cite{appx_convenet_2022} consists of two different blocks: the dense block is devoted to the parameter-dependent coefficient approximation, while the convolutional block strives to reconstruct the spatial basis function. Instead, POD-DL-ROMs compute the spatial basis before the training of neural networks by means of SVD \cite{Manzoni2016} or randomized SVD \cite{rsvd_2014} through an unsupervised learning criterion: in principle, this means that POD-DL-ROMs do not need any active weights to reconstruct the spatial basis functions, while the CNN approach needs $O(\varepsilon^{-\frac{2}{2\alpha-1}}\log(h^{-1}))$ weights to \textit{learn} them (we refer the reader to \textit{Theorem 2} in \cite{appx_convenet_2022}).
    \item In the decomposition employed in \cite{appx_convenet_2022}, $C$ plays the role of the reduced dimension: it is an analogue of the POD-dimension $N$ employed within the POD-DL-ROM technique. 
    In the following, we exploit an optimality result fulfilled by the POD basis to show that the complexity of the neural network in the parameter-to-coefficient map approximation is lower in the case of POD-DL-ROM when compared to the approach proposed in \cite{appx_convenet_2022}.
\end{itemize}

The quasi-optimality of the POD decomposition in its discrete formulation confirms that with a $N$-terms truncation, provided a sufficient amount of data have been suitably sampled, no linear decomposition captures as much variance as the discrete formulation of the POD decomposition, so that the reduced dimension $C$ of \cite{appx_convenet_2022} satisfies the inequality $C > N$ with probability $1 - \delta$ (see Subsection \ref{subsec:POD-analysis} and Appendix \ref{sec:appendix}). Furthermore, we assume that: 
 \begin{itemize}
     \item [(i)] $N \gg n$ as usual, since we expect that the nonlinear Kolmogorov $n$-width decays (much) faster than the linear reduced dimension $N$;  
     
     \item[(ii)] $u(\cdot,\muv,t) \in C^\alpha(\Omega)$ for any $ (\muv,t) \in \mathcal{P} \times \mathcal{T}$ for some $\alpha \ge 1$ to comply with the hypotheses of \textit{Theorem 2} of \cite{appx_convenet_2022};  
     
     \item[(iii)] the parameter-to-solution map has regularity $r$, i.e. $\mathcal{G} \in W^{r,\infty}(\mathcal{P} \times \mathcal{T}; \mathbb{R}^{N_h})$;
     
     \item[(iv)] the decoder map is adequately regular, namely $\frac{n}{s-1} > \frac{p+1}{r}$ ($s \ge 3r + 1$ is sufficient, as in \ref{subsec:POD_DNN}).
 \end{itemize}
We recall that \textit{Theorem 2} in \cite{appx_convenet_2022} provides the estimate $C = O(\varepsilon^{-\frac{2}{2\alpha-1}})$. Therefore, in the worst case scenario $N = O(\varepsilon^{-\frac{2}{2\alpha-1}})$; however, depending on the singular values decay that in some cases might be even exponential (e.g. stationary elliptic PDEs, analytic parameter-to-solution maps, see \cite{Manzoni2016}) we actually obtain improved estimates. We then derive:
\begin{equation*}
    \begin{aligned}
        w_{POD-DL-ROM} &= O(N \varepsilon^{-n/(s-1)} \log(\varepsilon^{-1})) + O(n \varepsilon^{-(p+1)} \log(\varepsilon^{-1})) \\ & \approx O(N \varepsilon^{-n/(s-1)} \log(\varepsilon^{-1})) \\ &
        \lesssim O(C \varepsilon^{-n/(s-1)} \log(\varepsilon^{-1}))
        \\ &= O(\varepsilon^{-\frac{2}{2\alpha-1} - \frac{n}{(s-1)}} \log(\varepsilon^{-1}))
        \\ &\lesssim O(\varepsilon^{-\frac{2}{2\alpha-1}} [\varepsilon^{- \frac{n}{(s-1)}} (\log(\varepsilon^{-1}) + \log(h^{-1})])
        \\ &\lesssim O(\varepsilon^{-\frac{2}{2\alpha-1}} [\varepsilon^{- \frac{p+1}{r}} (\log(\varepsilon^{-1}) + \log(h^{-1})])
        \\ &= w_{CNN}.
    \end{aligned}
\end{equation*}
Thus, we can conclude that, if the hypotheses setting is verified, the overall complexity of the POD-DL-ROMs in terms of active weights is lower (or equal) than the complexity of the CNN architecture proposed in \cite{appx_convenet_2022}.

\section{Numerical experiments}
\label{sec:six}
Within this section, we present different numerical tests, aiming at validating the theoretical analysis proposed in the previous Sections. In particular, we focus on {\em (i)} the error bounds of Theorems \ref{thm:lower_bound}--\ref{thm:upper_bound} and the error decomposition formula, as well as  on {\em (ii)}  the role of the reduced dimension $N$ and the total number of snapshots $N_{data}$ and on  {\em (iii)} the comparison against recent approaches proposed in the literature, in light of the theoretical results of Sections \ref{sec:four} and \ref{sec:five}. In particular, the numerical experiments involve:
\begin{itemize}
    \item[a)] a benchmark test case with an analytically defined operator that allows us to know \textit{a priori} the properties of the parametric operator (like, e.g., the regularity of the parameter-to-solution map) in order to validate the theoretical estimates on the network complexity;
    \item [b)] a linear 1D Initial Boundary Value Problem (IBVP), to show how to select $N_{data}$ and $N$ in order to minimize the \textit{a priori} error (given by the sum of $\mathcal{E}_S$ and $\mathcal{E}_{POD}$), then validating \textit{a posteriori} the network complexity as a function of the relative error; 
    \item [c)] a nonlinear 2D time-dependent IBVP in a non-conventional domain, to show the effectiveness of the POD-DL-ROM approach when dealing with more complex problems, validating also the \textit{lower bound} and the \textit{upper bound} on the relative error $\mathcal{E}_R$, which stem from the theoretical analysis.
\end{itemize}

We remark that the complexity analysis of POD-DL-ROM and related approaches is discussed from a theoretical point of view only in terms of the approximation error; however, when numerical experiments are addressed, we also have to take into account the training error, which plays a major role especially when the network is sufficiently deep or wide, or data are limited. For the same reasons, in our numerical experiments we mainly address the complexity study in terms of number of active weights $w$, since the latter is a quantity which is less sensitive (when compared to the depth $L$) to the training error. Thus, the experimental complexity analysis presented here may not reflect exactly the estimates provided in the previous sections, but they validate qualitatively the theory. However, within the present section, aiming at mitigating the effect of the training error on the error estimates, we employ several \textit{ad hoc} strategies, like, e.g., 
\begin{itemize}
    \item we employ early stopping to prevent overfitting;
    \item the approximation results in terms of network complexity are achieved in an error range $[\varepsilon_1, \varepsilon_2]$ that is deemed appropriate for the chosen number of samples $N_{data}$: in practice the training error depends on data availability;
    \item for fixed number of active weights, we regulate the network architecture trying to randomly achieve the configuration that minimizes the training error; we keep the depth of the network as low as possible in order to ensure convergence to a suitable minimum and avoid expensive training loops;
    \item starting from educated guesses, we look for the best training hyperparamenters (which are the learning rate and the learning rate decay).
\end{itemize}

Finally, we remark that, in order to comply with the hypotheses of the Theorems of  Section~\ref{sec:four}, we limit the numerical experiments to generic dense layers and employ LeakyReLU as activation function:
\begin{equation*}
    \textnormal{LeakyReLU}_\alpha(x) = \left\{ \begin{aligned} x, \qquad & x \ge 0 \\
    \alpha x, \qquad & x < 0\end{aligned} \right..
\end{equation*}
Unless otherwise stated, we set $\alpha = 0.1$. The optimization procedure is carried out by employing the Adam algorithm \cite{KingmaB14}.

\subsection{Benchmark test case}

We begin our experimental analysis by considering a benchmark test case similar to the one described in \cite{appx_convenet_2022}, and involving the reconstruction of an analytically defined operator, namely
\begin{equation*}
    u_\beta(x,\muv) = \mu_3 |x-\mu_1|^\beta e^{-\mu_2 x}, \qquad x \in [0,1],
\end{equation*}
where $\muv = [\mu_1, \mu_2, \mu_3] \in \mathcal{P} = [0,1] \times [0,1] \times [1,2]$. Within this numerical test we vary $\beta \in \{3/2,7/3,3\}$ and we analyze the three resulting cases independently. Notice that the hyperparameter $\beta > 0$ controls the regularity of the parameter-to-solution map. Indeed,
\begin{equation*}
    \begin{aligned}
    u_{3/2}(x,\cdot) &\in W^{1,+\infty}(\mathcal{P})  \setminus W^{2,+\infty}(\mathcal{P}) \\
    u_{7/3}(x,\cdot) &\in W^{2,+\infty}(\mathcal{P})  \setminus W^{3,+\infty}(\mathcal{P}) \\
    u_{3}(x,\cdot) &\in W^{3,+\infty}(\mathcal{P})  \setminus W^{4,+\infty}(\mathcal{P});
    \end{aligned}
\end{equation*}
thus $\beta=3/2,7/3,3$ correspond to $r=1,2,3$ respectively, where $r$ is defined as the regularity of the parameter-to-solution map in agreement with this paper notation.
Furthermore, the problem does not depend on the time variable, thus we set $N_t = 1$,  $N_{data} = N_s$ and $p=2$ (instead of $p=3$) to comply with the theoretical framework of the present work. 
Moreover, we discretize the problem in space by means of a uniform discretization with $N_h = 1000$.  Selecting $n = 5 \le 2 p + 3 = 7$ to ensure both a suitable compression and an adequate representation in the latent space, $N_{s} = 500$, and
\begin{equation*}
    N = N(r) = \left\{ \begin{aligned} 20, \qquad& r = 1 \\ 17, \qquad& r=2 \\ 15, \qquad& r = 3,\end{aligned} \right.
\end{equation*}
to control the variability retained by the preliminary linear dimensionality reduction.
We then proceed towards a complexity analysis, showing a comparison of the results against the CNN approach considered in \cite{appx_convenet_2022}, the POD+DNN framework and the lin+ResNets technique. We remark that for the sake of fairness and consistency, we keep the batch size during training equal to $B=20$ for every comparison considered in the benchmark test case. Then, for any $r \in \{1,2,3\}$, we estimate the approximation error $\mathcal{E}_R$ on the respective test set consisting of $N_s^{test} = 10^4$ samples.

From a theoretical standpoint, we immediately notice  $\mathcal{G}: \muv \mapsto \mathbf{u}(\muv) \in W^{r,+\infty}(\mathcal{P}; \mathbb{R}^{N_h})$; then, from the findings of Section~\ref{sec:five}, since $n \ll N$, we can infer that
\begin{equation*}
    \begin{aligned}
    w_{POD+DNN} &= O(N \varepsilon^{-3/r}\log(\varepsilon^{-1})) \\
    w_{POD-DL-ROM} &= O(N \varepsilon^{-5 / (s-1)}\log(\varepsilon^{-1})).
    \end{aligned}
\end{equation*}
Thus, owing to the fact that in the POD-DL-ROMs approach the perfect embedding Assumption with coefficients $s,s'$ is enforced thanks to their peculiar loss formulation, we expect them yielding a less steep increase (when compared to POD+DNNs) in the model complexity as the accuracy level decreases whenever the decoder map is suitably regular, which is equivalent to require $s > \frac{5}{3}r + 1$. Figure \ref{fig:Fig1} demonstrates that the latter behavior is more likely to happen as the regularity of the parameter-to-solution map $r$ decreases.
\begin{figure}[t!]
    \centering
    \vspace{-0.15cm}
\includegraphics[width=0.99\textwidth]{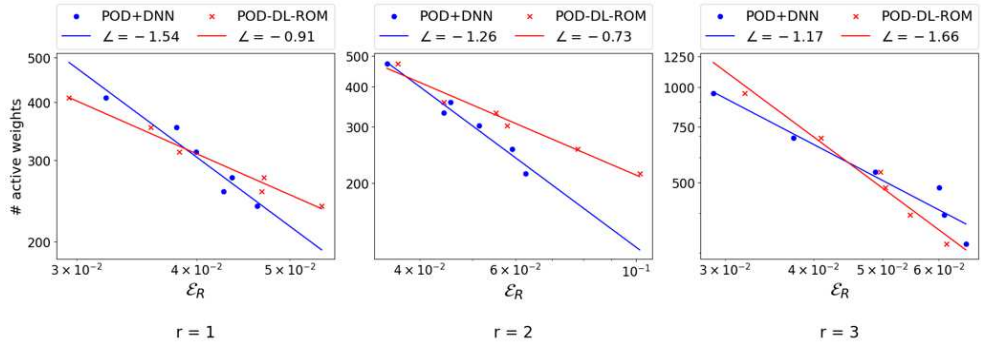}
    \caption{Benchmark test case: model complexity comparison between POD-DL-ROMs and POD+DNNs as the parameter-to-solution regularity $r$ varies in $\{1,2,3\}$. The trends are displayed through solid lines, which fit the collected results in the least squares sense.}%
    \label{fig:Fig1}%
     \vspace{-0.15cm}
\end{figure}

We then compare POD-DL-ROMs against the lin+ResNets approach; for the latter, we limit the analysis to the case where the basis functions are yielded by POD for the sake of consistency. We thus fix the latent space dimension of the residual layers as $k = 5$ and, from the estimates obtained in Section~\ref{sec:five}, we recall that the complexity bound of lin+ResNets in terms of number of active weights is in general independent of the regularity of the parameter-to-solution map, namely:
\begin{equation*}
    w_{lin+ResNets} = O(Nk\varepsilon^{-1}).
\end{equation*}
We thus remark that the lin+ResNets approach does not take advantage of any regularity assumption on the parameter-to-solution map: we then expect a similar trend as $r$ varies in $\{1,2,3\}$. Nonetheless, if the trained POD-DL-ROM architecture are able to find an adequate representation in the latent space which induces a very regular decoder, that is $s > 6$, we can ensure that the POD-DL-ROM outperform the lin+ResNets approach in terms of complexity: this behavior is indeed observed in Figure \ref{fig:Fig2}.

\begin{figure}[ht]
    \centering
    \vspace{-0.15cm}
\includegraphics[width=0.99\textwidth]{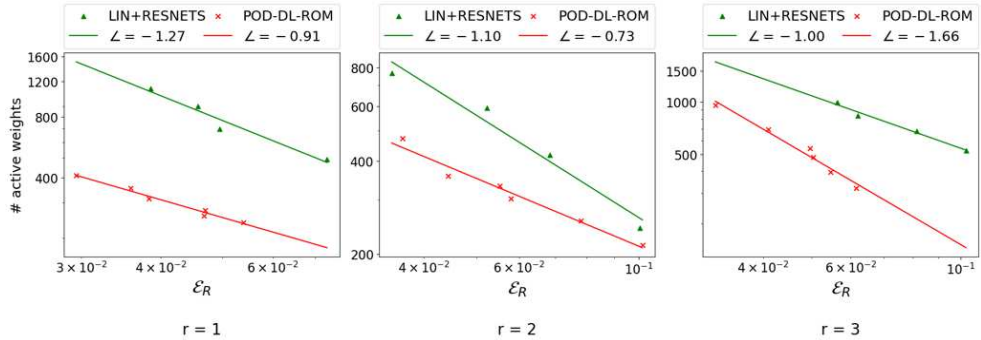}
    \caption{Benchmark test case: model complexity trend of POD-DL-ROMs and the lin+ResNets approach for different values regularity of the parameter solution map $r$. }%
    \label{fig:Fig2}%
    \vspace{-0.15cm}
\end{figure}

Finally, we consider the comparison against the CNN approach considered in \cite{appx_convenet_2022}: if the decoder map is sufficiently regular (from the theoretical analysis we derive the condition $s \ge \frac{5}{3}r + 1$), POD-DL-ROMs take advantage of the basis optimality to achieve a less steep increase of complexity as the error bound $\mathcal{E}_R < \epsilon$ decreases: the behavior is indeed observed in Figure \ref{fig:Fig3}, in the cases when the regularity of the parameter-to-solution map is low ($r=1,2$). Moreover, differently from the CNN-based technique, we remark that the POD-DL-ROMs' algorithm does not require to learn the basis functions, thus not affecting the overall complexity of the underlying network.

\begin{figure}[ht]
    \centering
    \vspace{-0.15cm}
\includegraphics[width=0.99\textwidth]{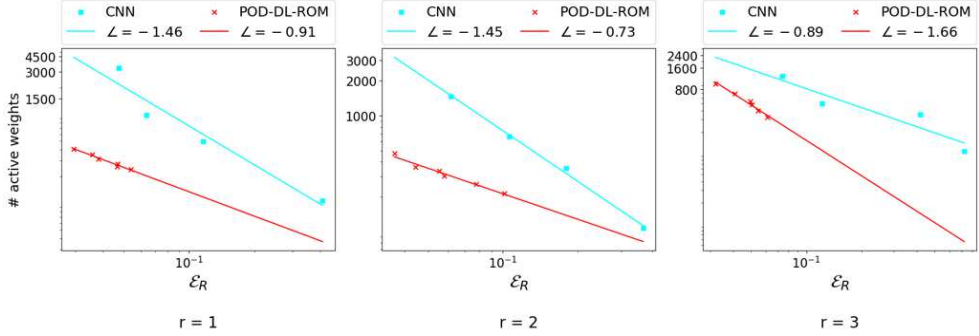}
    \caption{Benchmark test case: comparison between POD-DL-ROMs and CNNs in terms of number of active weights, varying the regularity $r \in \{1,2,3\}$. }%
    \label{fig:Fig3}%
    \vspace{-0.15cm}
\end{figure}

\subsection{1D Initial Boundary Value Problem}

The present test case is designed to highlight the advantages of POD-DL-ROMs when compared to other considered approaches even when dealing with time-dependent parametrized problems. Moreover, before starting the training process, we show \textit{a priori} how to choose the hyperparameters $N,N_s, N_t$, based on the analysis of $\mathcal{E}_S$ and $\mathcal{E}_{POD}$. In particular, we consider the following IBVP: 
\begin{equation*}
    \left\{
    \begin{aligned}
        &\frac{\partial{u}}{\partial{t}} - \frac{\partial^2{u}}{\partial{x}^2} = u + 10\cos(x)\sin(2\pi t) , &\mbox{in} \ (0,\pi) \times (0,T] \\
        &u = 10(2 \mu^3 - 3 \mu^2 + \mu) , & \mbox{at} \ \{x=0\} \times (0,T]\\
        &\frac{\partial{u}}{\partial{x}} = 2 |1 - 2 \mu| - 1 , & \mbox{at} \  \{x =\pi\} \times (0,T] \\
        &u(x,0)  = u_0(\mu),&  \mbox{in} \ (0,\pi),
    \end{aligned}
    \right. 
\end{equation*}
where the initial condition is
\begin{equation*}
    u_0 = u_0(\mu) = 10(2 \mu^3 - 3 \mu^2 + \mu) \cos(x) + (2 |1 - 2 \mu| - 1) \sin(x),
\end{equation*}
while $\mu \in \mathcal{P} = [0,1]$ and $T = 1$. Thus, $p = 1$ and we can fix $n = 5 = 2p + 3$ to ensure an adequate representation in the latent space, according to the framework presented in the present paper. We collected synthetic data generated with an high-fidelity model solved on a uniform grid of $N_h = 100$ points: we generate a test set of $N_s^{test} = 100$ samples of $N_t^{test} = 200$ snapshots each with a \textsc{Matlab}-based PDE solver, sampling $\muv \sim \mathcal{U}(\mathcal{P})$ iid and $t$ from a uniform grid of step $\Delta t^{test} = T / N_t^{test}$.

\begin{figure}[b!]
    \centering
    \vspace{-0.15cm}
\includegraphics[width=0.99\textwidth]{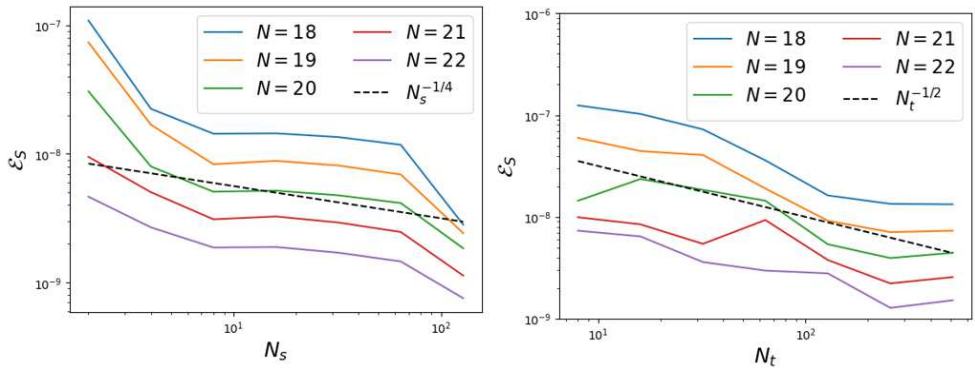}
    \caption{1D IBVP test case: decay of the sampling error $\mathcal{E}_S$ with respect to $N_s$, $N_t$ and $N$.}%
    \label{fig:Fig4}%
    \vspace{-0.15cm}
\end{figure}

\begin{figure}[ht]
    \centering
    \vspace{-0.15cm}
    \includegraphics[width=11.5cm]{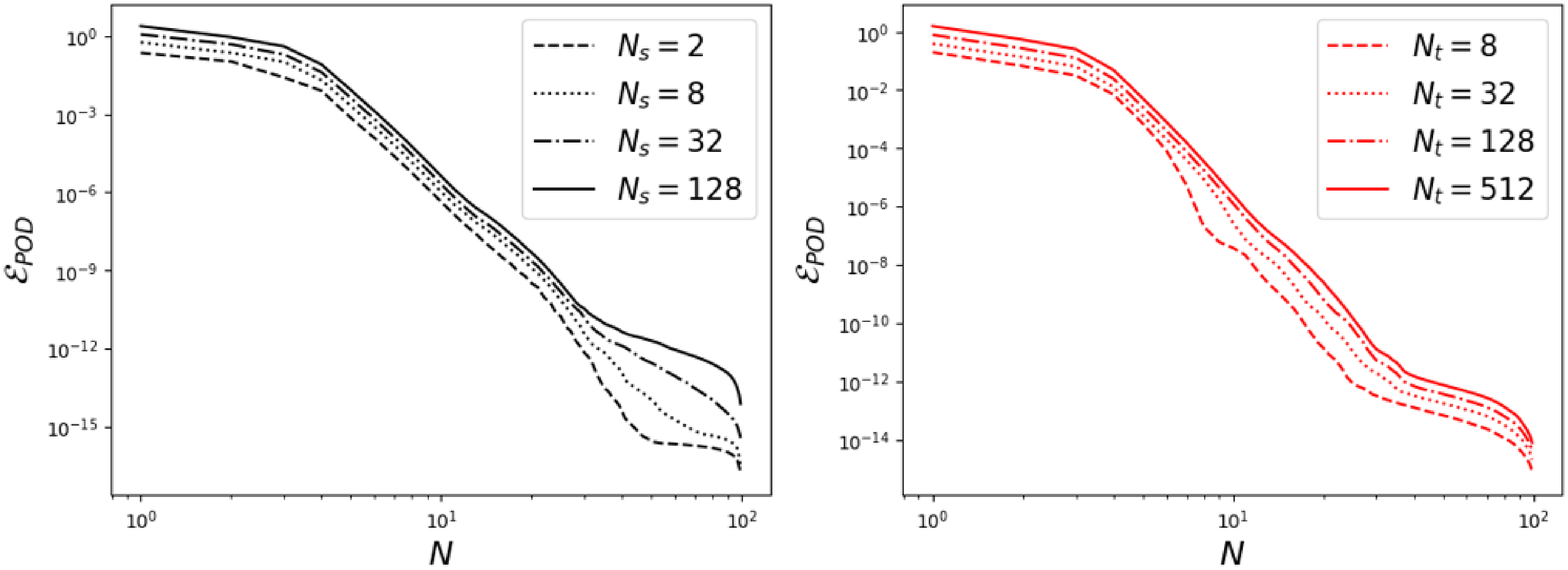}
    \caption{1D IBVP test case: decay of the projection error $\mathcal{E}_{POD}$ varying $N_s$, $N_t$ and $N$.}%
    \label{fig:Fig5}%
    \vspace{-0.15cm}
\end{figure}

We start by analyzing the dependence of $\mathcal{E}_{S}$ on $N_s$, $N_t$ and $N$; for the sake of clarity, we specify that the sampling criterion employed in the \textit{a priori} analysis below is based on the theoretical analysis of the entire work: thus, we assume $\muv \sim \mathcal{U}(\mathcal{P})$ iid and that $t$ is sampled from a uniform grid of step $\Delta t = 1 / N_t$. To analyze the effect of $N_s$ on the sampling error, we fix $N_t = 1000$ and we generate a group of datasets depending on $N_s \in \{l = 2^k: k=1,\ldots,7\}$: as shown in Figure \ref{fig:Fig4}, the decay has slope $-1/4$ and it is independent of the chosen value of $N$. Conversely, we fix $N_s = 100$ and vary $N_t \in \{l = 2^k: k=3,\ldots,9\}$, validating experimentally in Figure \ref{fig:Fig4} that $\mathcal{E}_S \sim N_t^{-1/2}$, independently of $N$. 
We then move to the analysis of the projection error, showing in Figure \ref{fig:Fig5} how $\mathcal{E}_{POD}$ decays with $N$ and is mostly independent of $N_s$ and $N_t$ respectively.
We notice that the present analysis is done before the training of the underlying neural network and allow us to know \textit{a priori} how much variance is not accounted for due to the sampling ($\mathcal{E}_S$) and the initial dimensionality reduction ($\mathcal{E}_{POD}$), allowing us to calibrate the values $N, N_s, N_t$ before we start the expensive training procedure. The idea is to choose $N,N_s, N_t$ to guarantee that $\mathcal{E}_{POD}$ and $\mathcal{E}_S$ are suitably small, so that we can control the relative error $\mathcal{E}_R$ with a strict bound, which is provided by the error decomposition of Theorem \ref{thm:error_decomposition}. Thus, based on the results of the present \textit{a priori} analysis, we choose $N_s = 50$, $N_t = 20$, $N = 20$.

\begin{figure}[ht]
    \centering
    \vspace{-0.15cm}
\includegraphics[width=0.99\textwidth]{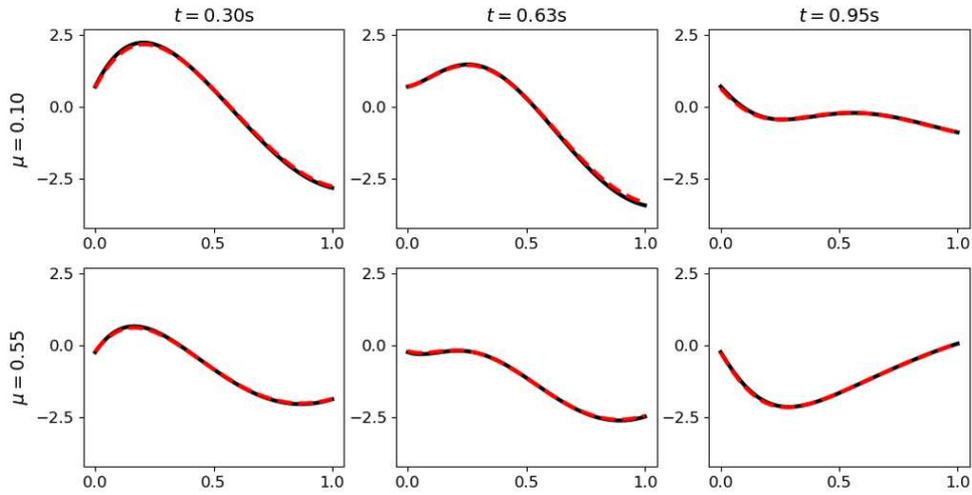}
    \caption{1D IBVP test case: comparison between the "true" solution (solid black line) and the most accurate POD-DL-ROM prediction (dashed red line) to demonstrate that the variability of the solution manifold is correctly reproduced.}%
    \label{fig:Fig6}%
    \vspace{-0.15cm}
\end{figure}

We then move our focus to the comparison of the POD-DL-ROM technique against other approaches in terms of complexity, showing the relation between the relative error $\mathcal{E}_R$ and the number of active weights employed in the underlying neural network. Notice that, since the analytical solution of the IBVP is not available, here we are not provided with any information on the regularity of the parameter-to-solution map. Anyway, experimental results on the complexity analysis confirm our theoretical expectations: when dealing with parameter-to-solution maps arising from parametric PDEs, POD-DL-ROMs' complexity increases slower than POD+DNNs' one as the relative error decreases. Indeed, the latent representation of the POD-DL-ROM approach induces a decoder that is extremely regular, that is $s \gg 2$, which enables a slow increase in network complexity, as suggested by the theoretical approximation bounds of Theorem \ref{thm:upper_bound} and validated in Figure \ref{fig:Fig7}. Similarly, we notice that the results relative to the comparison between POD-DL-ROMs and lin+ResNets are in agreement with the theory, demonstrating again how, lin+ResNets are outperformed in terms of complexity by POD-DL-ROMs, when it is possible for the latter to achieve an extremely regular decoder map due to an adequate latent representation.
Finally, when compared to the Fourier-inspired CNN technique POD-DL-ROMs' number of active weights show a slower increase as the relative error $\mathcal{E}_R$ decreases, as shown in Figure \ref{fig:Fig7}; as proved theoretically in Section~\ref{sec:five}, the magnitude of the slope is strongly linked to the optimality of the basis functions. Moreover we validate how the burden of learning the set of basis function impacts heavily on the underlying CNN complexity, which shows a remarkable difference when compared the POD-DL-ROM approach in terms of number of active weights, not only regarding the slope magnitude but also in the absolute sense. The observed behavior highlights how crucial it is in terms of complexity to consider a \textit{fixed} set of optimal basis functions instead of a \textit{learnable} set of non-optimal ones. 

Thus, this validates the theoretical considerations and concludes our comparison based on model complexity, demonstrating how POD-DL-ROMs outperform any of the considered techniques when tackling more complex problems, for which the regularity of the parameter-to-solution map is low or unknown \textit{a priori}.

\begin{figure}[b!h]
    \centering
    \vspace{-0.15cm}
\includegraphics[width=0.99\textwidth]{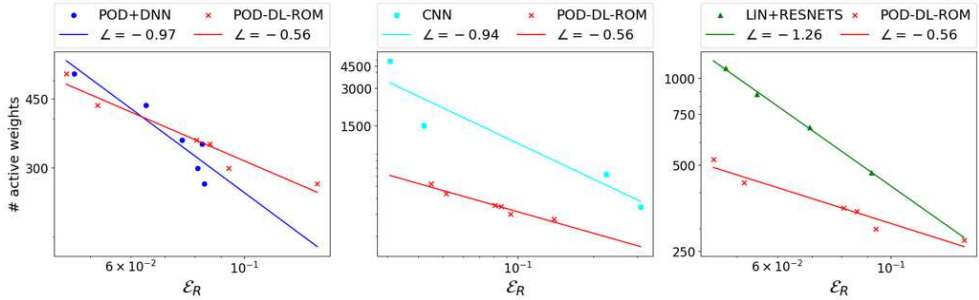}
    \caption{1D IBVP test case: comparison between POD-DL-ROMs and other techniques in terms of number of active weights. The solid line represents the least squares fitting of the \textit{log-log} data.}%
    \label{fig:Fig7}%
    \vspace{-0.15cm}
\end{figure}

\subsection{2D nonlinear Initial Boundary Value Problem}
The last test case involves a nonlinear version of a time-dependent nonlinear parametrized diffusion equation with a non-affine source term in an unconventional domain; the strong formulation of the problem at hand takes the form
\begin{equation*}
    \left\{
    \begin{aligned}
        &\frac{\partial{u}}{\partial{t}} -  \nabla \cdot \biggl(0.001 (1+u^2) \nabla u\biggr) = 0,& \mbox{in} \ \Omega \times (0,T] \\
        &u = 1 - e^{-100 t} + h(x,y,\mu)  e^{-100 t}, & \mbox{on} \  \Gamma_D \times (0,T] \\
        &\frac{\partial{u}}{\partial{n}} = 0, & \mbox{on} \  \Gamma_N \times (0,T] \\
        &u_0 = h(x,y,\mu), & \mbox{in} \  \Omega,
    \end{aligned}
    \right.
\end{equation*}
where $T=0.05$ and 
\begin{itemize}
    \item $h(x,y,\mu) = 0.1 + 10 y \sin(\mu \pi x)$ represents a non-affine term, being $\mu \in  \mathcal{P} = [5,7]$ the parameter that regulates the spatial frequency of $h = h(x,y, \mu)$;
    \item letting $E_{a,b}(x,y)$ be the ellipse of axes $a$ and $b$ and center $(x,y)$, we set $D_1 = E_{0.2,0.2}(0.5,0.4)$ and $D_2 = E_{0.3,0.1}(1.0,0.2)$; then, we can define the domain as $\Omega = (0,1) \times (0,0.4) \setminus (D_1 \cup D_2)$;
    \item the Dirichlet and the Neumann boundary are  $\Gamma_D = \partial D_1 \cup \partial D_2$ and $\Gamma_N = \partial \Omega  \setminus (\partial D_1 \cup \partial D_2)$, respectively.
\end{itemize}
Through this numerical experiment we aim at verifying the \textit{upper bound} and \textit{lower bound} results presented in Section~\ref{sec:four}. 
To do so, we generate the training set and the test set input-output pairs through the numerical solution of the discretized problem on a mesh of $N_h = 1666$ dofs by means of P1-FEM, employing a Forward Euler time-advancing scheme and the Newton method to handle nonlinearities. The training set is made by $N_s=20$ samples relative to $\muv \sim \mathcal{U}(\mathcal{P})$ iid of $N_t = 30$ snapshots each, sampling $t$ from a uniformly space time grid of step $T/N_t$. The test set data consist of $N_s^{test} = 30$ samples, evaluated on the same time grid employed in the training set.

\begin{figure}[ht]
    \centering
    \vspace{-0.15cm}
\includegraphics[width=0.99\textwidth]{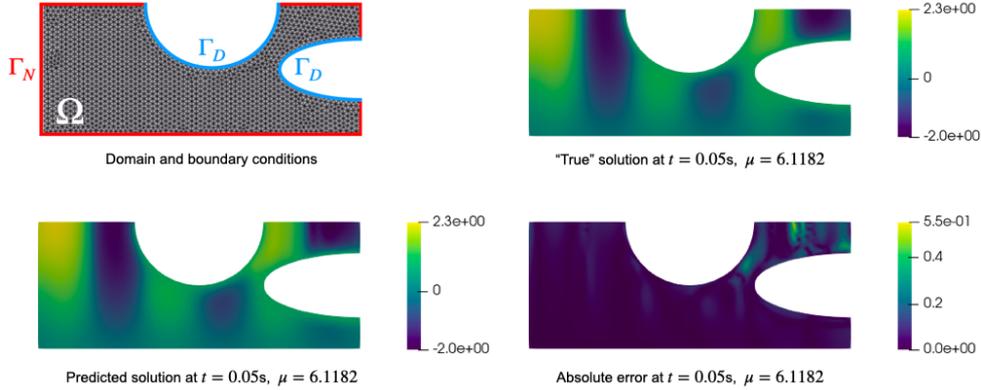}
    \caption{2D IBVP test case: domain and boundary specifics (upper left), comparison between "true" solution (upper right) and POD-DL-ROM's predicted solution (lower left) and visualization of the absolute error (lower right), in the case of $N=32$.}%
    \label{fig:Fig8}%
    \vspace{-0.15cm}
\end{figure}

Then, for each $N \in \{2^k, k=0,\ldots,5\}$ we train a POD-DL-ROM of latent dimension $n=2p+1=5$, which is composed of:
\begin{itemize}
    \item a reduced network of $3$ hidden dense layers of $10$ units each;
    \item an encoder and a decoder with $5$ hidden dense layers of $25$ units each.
\end{itemize} 
We then evaluate the lower bound $\frac{m}{M}\tilde{\mathcal{E}}_{POD}$, the upper bound due to the error decomposition formula $\mathcal{E}_{NN} + \mathcal{E}_S + \mathcal{E}_{POD}$, the value relative error $\mathcal{E}_R$, according to the theoretical framework of Section~\ref{sec:four}.

\begin{figure}[ht]
    \centering
    \includegraphics[width=0.95\textwidth]{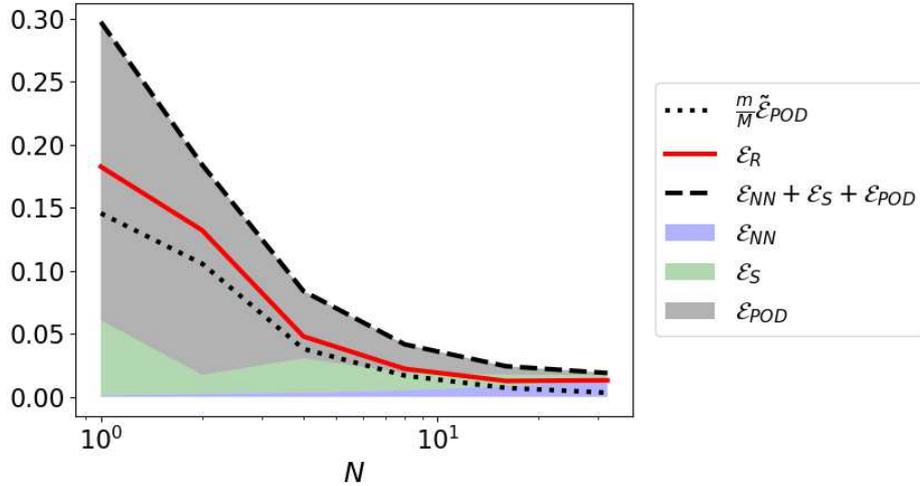}
    \caption{2D IBVP test case: error bounds analysis varying the POD dimension $N$}%
    \label{fig:Fig9}%
\end{figure}

We show both the lower bound and the upper bound results in Figure \ref{fig:Fig9}, displaying as well the error contributions $\mathcal{E}_{NN}, \mathcal{E}_S, \mathcal{E}_{POD}$ to assess the way they affect the relative error $\mathcal{E}_R$. We then remark again that it is crucial for POD-DL-ROMs to provide both an adequate neural network approximation of the parameter-to-solution map and a suitably large POD dimension. Indeed, we notice that in the present test case, especially for low values of $N$, $\mathcal{E}_{NN}$ shows a marginal contribution to the upper bound value when compared to the sampling error $\mathcal{E}_S$ and the projection error $\mathcal{E}_{POD}$. Furthermore, as expected, we observe the strong dependence of the lower bound $\frac{m}{M}\tilde{\mathcal{E}}_{POD}$ on the POD dimension, demonstrating again the importance of choosing an adequate value for $N$. Finally, we assess \textit{a posteriori} that the number of samples in the training set is suitable since the sampling error $\mathcal{E}_S$ does not heavily influence the upper bound of the relative error.

\section*{Conclusions}
\label{sec:seven}

The main goal of this work is to suggest effective and practical strategies to set a POD-DL-ROM stemming from a rigorous analysis of the technique, to control the approximation accuracy, measured in terms of the relative error $\mathcal{E}_R$, which is linked to relevant features and hyperparameters that can be effectively regulated. To accomplish the task, we analyze the error $\mathcal{E}_R$, providing a \textit{lower bound} that depends only on the projection-based nature of the method. Then, by the \textit{error decomposition} formula and the \textit{upper bound} result, we highlight the contribution of sampling, POD projection and neural network approximation; in particular:
\begin{itemize}
    \item [(i)] on the basis of the analysis of the sampling error $\mathcal{E}_S$ we propose a family of strategies to adopt in the data collection phase in order to ensure the convergence of $\mathcal{E}_S \rightarrow 0$ in the limit of infinite data, providing also a decay estimate through Monte Carlo analysis in terms of the number of sampled snapshots $N_{data}$;
    
    \item [(ii)] we determine a practical criterion based on the eigenvalue decay to control $\mathcal{E}_{POD}$ in terms of the reduced dimension $N$;
    
    \item [(iii)] starting from  the approximation results proposed in \cite{yarotsky2017}, we estimate the complexity of the underlying neural network that is required to reach a given accuracy.
\end{itemize}

Then, relying on the aforementioned findings, we compare the POD-DL-ROM paradigm to other architectures that are widely used in the literature, namely DL-ROMs \cite{dlrom_bounds,dlrom_2021,dlrom_electrophys}, POD+DNNs \cite{chen_2021,hesthaven_2018,salvador_2021}, POD-DeepONets \cite{deeponet_vs_fno_2022}, lin+ResNets \cite{oleary-2022} as well as CNNs \cite{appx_convenet_2022}, showing the strengths of the POD-DL-ROM strategy, especially when dealing with low-regularity maps. Ultimately, we demonstrate the outstanding approximation properties of POD-DL-ROMs, which motivate the excellent performance already encountered in a variety of test cases analyzed in the recent literature \cite{pod_dl_rom_2022,POD-DL-ROM-fluids} and in the present work.
Several working directions could stem from the present paper; for instance, more efficient sampling criteria arising from Monte Carlo analysis could be implemented: we mention variance reduction techniques and Quasi Monte Carlo methods \cite{caflisch_1998}, among others. On the other hand, one could consider \textit{ad hoc} layers to be employed in the reconstruction of parameter-to-POD-coefficients maps instead of relying purely on dense layers; however, this latter option would require novel and precise approximation results for the considered layers. Moreover, an alternative formulation could split the time- and the parameter-dependence, avoiding to treat time as an additional parameter, similarly to what has been proposed in \cite{mionet_2022}, in order to further enhance the   approximation bounds proposed in this paper.

%\backmatter

%\bmhead{Supplementary information}

%If your article has accompanying supplementary file/s please state so here. 

%Authors reporting data from electrophoretic gels and blots should supply the full unprocessed scans for key as part of their Supplementary information. This may be requested by the editorial team/s if it is missing.

%Please refer to Journal-level guidance for any specific requirements.

\section*{Acknowledgments}

We acknowledge the support of Fondazione Cariplo, Italy, Grant n. 2019-4608, of the PNRR-PE-AI FAIR project funded by the NextGeneration EU program, as well as of
the National Group of Scientific Computing (GNCS) of INDAM - Istituto Nazionale di Alta Matematica. SF also acknowledges the Isaac Newton Institute for Mathematical Sciences, Cambridge, UK, for support and hospitality during the programme “The mathematical and statistical foundation of future data-driven engineering”, EPSRC grant no EP/R014604,  where part of this work was undertaken.

\begin{appendices}

\section{Additional proofs}
\label{sec:appendix}

\subsection{Proof of Proposition \ref{prop:norm}}
We notice immediately that the integral is well defined $\forall \mathbf{v} \in L^2(\mathcal{P} \times \mathcal{T}; \mathbb{R}^{N_h})$ thanks to the boundedness assumptions on the solution $\mathbf{u} \in L^2(\mathcal{P} \times \mathcal{T}; \mathbb{R}^{N_h})$. We also remark that the boundedness hypotheses may be relaxed: our choice was aimed at consistency with the other theoretical results of the present work.
In order to prove that $\|\cdot\|_{L^2_w}$ is a norm, we have to show that:
\begin{itemize}
    \item [(i)] It satisfies the triangle inequality. Given $\mathbf{v}, \mathbf{z} \in L^2(\mathcal{P}\times \mathcal{T}; \mathbb{R}^{N_h})$, by means of the triangular inequality, it is trivial to show that
    \begin{equation*}
        \begin{aligned}
        &\|\mathbf{v} + \mathbf{z}\|^2_{L^2_w} =
        \\ & = 
        \int_{\mathcal{P} \times \mathcal{T}}\|\mathbf{v}(\muv,t) + \mathbf{z}(\muv,t)\|^2 w(\muv,t) d(\muv,t) \le \\ &\le
        \int_{\mathcal{P} \times \mathcal{T}} (\|\mathbf{v}(\muv,t) \| + \|\mathbf{z}(\muv,t)\|)^2 w(\muv,t) d(\muv,t) = \\ & =
        \int_{\mathcal{P} \times \mathcal{T}} (\|\mathbf{v}(\muv,t) \|^2 + \|\mathbf{z}(\muv,t)\|^2 +  2\|\mathbf{v}(\muv,t) \| \|\mathbf{z}(\muv,t) \| )w(\muv,t) d(\muv,t).
        \end{aligned}
    \end{equation*}
    \noindent Moreover, by the Cauchy-Schwarz inequality, the following inequality holds,
    \begin{equation*}
        \begin{aligned}
        &\int_{\mathcal{P} \times \mathcal{T}} \|\mathbf{v}(\muv,t) \| \|\mathbf{z}(\muv,t) \| w(\muv,t) d(\muv,t) \le \\ &\le \sqrt{\int_{\mathcal{P} \times \mathcal{T}} \|\mathbf{v}(\muv,t) \|^2 w(\muv,t) d(\muv,t) \int_{\mathcal{P} \times \mathcal{T}} \|\mathbf{z}(\muv,t) \|^2 w(\muv,t) d(\muv,t)}.
        \end{aligned}
    \end{equation*}
    Thus, we can infer 
    \begin{equation*}
        \|\mathbf{v} + \mathbf{z}\|^2_{L^2_w} \le \|\mathbf{v}\|^2_{L^2_w} + \|\mathbf{z}\|^2_{L^2_w} + 2\|\mathbf{v}\|_{L^2_w} \|\mathbf{z}\|_{L^2_w}=(\|\mathbf{v}\|_{L^2_w} + \|\mathbf{z}\|_{L^2_w})^2
    \end{equation*}
    and derive the thesis;
    \item [(ii)] $\|\cdot\|_{L^2_w}$ is homogeneous thanks to the linearity of the integral;
    \item[(iii)] If $\mathbf{v} \in L^2(\mathcal{P} \times \mathcal{T}; \mathbb{R}^{N_h})$, $\|\mathbf{v}\|_{L^2_w} = 0$ implies that $\mathbf{v} = \bm{0}$ a.e. by trivial arguments. 
\end{itemize}

\subsection{Proof of Proposition \ref{prop:monte-carlo}}
Thanks to Assumption \ref{assumption:sampling}, trivially we obtain $\Delta t = T N_t^{-1} = O(N_t^{-1})$ and we set $t_i = i\Delta t$. Letting $f=f(\muv,t)$ be the (sufficiently regular) integrand of the integral that we want to approximate, we obtain
\begin{equation*}
    \mathbb{E}\biggl| \int_{\mathcal{P} \times \mathcal{T}} f(\muv,t) d(\muv,t) - \frac{\Delta t |\mathcal{P}|}{N_s} \sum_{i=1}^{N_t}\sum_{j=1}^{N_s} f(\muv_j, t_i) \biggr| \le I_1 + I_2,
\end{equation*}
where 
\begin{equation*}
    I_1 = \biggl| \int_{\mathcal{P} \times \mathcal{T}} f(\muv,t) d(\muv,t) - \Delta t \sum_{i=1}^{N_t} \int_{\mathcal{P}} f(\muv,t_i) d\muv \biggr| = O(N_t^{-1}) 
\end{equation*}
and
\begin{equation*}
    \begin{aligned}
    I_2 &= \Delta t \sum_{i=1}^{N_t} \mathbb{E} \biggr|\int_{\mathcal{P}} f(\muv,t_i) d\muv - \frac{|\mathcal{P}|}{N_s}\sum_{j=1}^{N_s} f(\muv_j,t_i) \biggr| \\ &= O\biggl(N_s^{-1/2} \Delta t \sum_{i=1}^{N_t} (\Var(f(\mu,t_i)))^{1/2}\biggr) \\ &= O\biggl(N_s^{-1/2} \biggl(O(N_t^{-1}) + \int_{\mathcal{T}} \Var(f(\muv,t)) dt\biggr)\biggr) = O(N_s^{1/2})
    \end{aligned}
\end{equation*}
Notice that
\begin{equation*}
    \int_{\mathcal{T}} \Var(f(\muv,t)) < +\infty
\end{equation*}
because
\begin{equation*}
    \int_{\mathcal{T}} \biggl( \int_{\mathcal{P}} f(\muv,t)^2 d\muv \biggr)^{1/2} dt \le T^{1/2} \biggr(\int_{\mathcal{T} \times \mathcal{P}} f(\muv,t)^2 d(\muv,t)\biggr)^{1/2} < +\infty,
\end{equation*}
since $f \in L^2(\mathcal{P} \times \mathcal{T})$.
Thus, the error we commit in approximating the integral goes to zero upon requiring $N_s,N_t \rightarrow \infty$. Finally, notice that
\begin{equation*}
    \frac{\Delta t |\mathcal{P}|}{N_s} = \frac{T |\mathcal{P}|}{N_{data}} = \frac{|\mathcal{P} \times \mathcal{T}|}{N_{data}}, 
\end{equation*}
which allows us to write
\begin{equation*}
    \mathbb{E}\biggl| \int_{\mathcal{P} \times \mathcal{T}} f(\muv,t) d(\muv,t) - \frac{|\mathcal{P} \times \mathcal{T}|}{N_{data}} \sum_{i=1}^{N_t}\sum_{j=1}^{N_s} f(\muv_j, t_i) \biggr| \le O(N_s^{-1/2} + N_t^{-1}).
\end{equation*}

\subsection{Quasi-optimality of the discrete formulation of the POD decomposition}

We base the following analysis on the results of the $(\mathcal{P} \times \mathcal{T})$-continuous problem proposed in \cite{Manzoni2016}. We first recall that by definition $\mathbf{V}_\infty \in \mathbb{R}^{N_h \times N}$ (where N is the POD dimension) is optimal for the $(\mathcal{P} \times \mathcal{T})$-continuous formulation, that is with respect to the $L^2(\mathcal{P} \times \mathcal{T}; \mathbb{R}^{N_h})$ norm. Formally, we set $\delta,\varepsilon > 0$ and, by assuming $\mathbf{u}(\muv,t) \in L^2(\mathcal{P} \times \mathcal{T},\mathbb{R}^{N_h})$, we define $T:L^2(\mathcal{P} \times \mathcal{T}) \rightarrow \mathbb{R}^{N_h}$ as
    \begin{equation*}
        Tg := \int_{\mathcal{P} \times \mathcal{T}} \mathbf{u}(\muv,t) g(\muv,t) d(\muv,t) \quad \forall g \in L^2(\mathcal{P} \times \mathcal{T}).
    \end{equation*}
The adjoint operator of $T$, namely $T^*$, enjoys the property
\begin{equation*}
    T^*\mathbf{w} = (\mathbf{u}(\muv, t),\mathbf{w})_2 \qquad \forall \mathbf{w} \in \mathbb{R}^{N_h}.
\end{equation*}
Moreover, recall the definition of the (continuous) %e define the 
correlation matrix \eqref{eq:correlation_matrix} and denote by  %$\mathbf{K}_\infty = \int_{\mathcal{P} \times \mathcal{T}} \mathbf{u}(\muv,t) \mathbf{u}(\muv,t)^T d(\muv,t) \in \mathbb{R}^{N_h \times N_h}$ and let 
$(\sigma_{k,\infty}^2,\bm{\zeta}_k)$ its eigenpairs (where $\{\bm{\zeta}_k\}_{k}$ denotes an orthonormal basis).
We thus define the HS-norm of $T$ as
\begin{equation*}
    \|T\|_{HS} = \sqrt{\sum_{k\le \rank(T)} \sigma_{k,\infty}^2}.
\end{equation*}
Setting 
\begin{equation*}
    \bm{\xi}_k = \frac{1}{\sigma_{k,\infty}} T^*\bm{\zeta}_k \qquad \forall k=1,\ldots,N_h,
\end{equation*}
we denote by $T_{N,\infty}$ the rank-$N$ Schmidt approximation, with
\begin{equation*}
    T_{N,\infty} = \sum_{k=1}^N \sigma_{k,\infty} \bm{\zeta}_k (\bm{\xi}_k(\muv,t),\cdot)_{L^2(\mathcal{P} \times \mathcal{T})} = \mathbf{V}_\infty\mathbf{V}_\infty^T T.
\end{equation*}
and by $T_{N} = \mathbf{V} \mathbf{V}^T T$ its approximation by means of the discrete POD formulation. 
%\begin{equation*}
 %   T_{N} = \mathbf{V} \mathbf{V}^T T.
%\end{equation*}
\textit{Theorem 6.2} and \textit{Proposition 6.3} in \cite{Manzoni2016} show that the rank-$N$ Schmidt operator and therefore the set of basis $\mathbf{V}_{\infty}$ are optimal with respect to the HS-norm, namely they retain the most variability. Formally:
\begin{equation}
\label{eq:HS-norm-min}
    \begin{aligned}
    &\|T_{N,\infty} - T\|_{HS} =  \min_{B \in \mathcal{B}_N} \|B - T\|_{HS}  \\ &= \min_{\mathbf{W} \in \mathbb{R}^{N_h \times N}: \mathbf{W}^T \mathbf{W} = \bm{I}} \biggl(\int_{\mathcal{P} \times \mathcal{T}} \|\mathbf{u}(\muv, t)- \mathbf{W}\mathbf{W}^T\bm{q}(\muv, t)\|^2 d(\muv,t)\biggr)^{1/2}  \\ &= \sqrt{\sum_{k>N} \sigma_{k,\infty}^2} \\ &= m \mathcal{E}_{POD,\infty},
    \end{aligned}
\end{equation}
where $\mathcal{B}_N = \{ B \in \mathcal{L}(L^2(\mathcal{P} \times \mathcal{T}); \mathbb{R}^{N_h})): \rank(B) \le N \land \|B\|_{HS} < +\infty\}$, being $\mathcal{L}(U)$ the space of linear continuous operators from $U$ to $U$, for $U$ Banach. 
Now, suppose to define $B_N \in \mathcal{B}_N$ which does not attain the minimum in \eqref{eq:HS-norm-min}, thus
\begin{equation}
\label{eq:HS-min-not-att}
     0 < 2\varepsilon_{max} := \|B_{N} - T\|_{HS} - \|T_{N,\infty} - T\|_{HS}.
\end{equation}
By means of the results of Theorem \ref{thm:error_decomposition}, with the same hypotheses, we have that
\begin{equation*}
\begin{aligned}
     \|T_{N,\infty} - T\|_{HS} &\le \|T_{N} - T\|_{HS} \\ &\le m(\mathcal{E}_S + \mathcal{E}_{POD}) \xrightarrow[N_s,N_t \rightarrow\infty]{a.s.} m\mathcal{E}_{POD,\infty} = \|T_{N,\infty} - T\|_{HS}.
\end{aligned}
\end{equation*}
Thus, since a.s. convergence implies convergence in probability, we derive that
\begin{equation*}
\begin{aligned}
     &\forall \delta > 0, \quad \forall 0 < \varepsilon < \varepsilon_{max}, \quad \exists N_s, N_t: \\
     &\hspace{4cm} \mathbb{P} \biggl\{\|T_{N} - T\|_{HS} - \|T_{N,\infty} - T\|_{HS} < \varepsilon \biggr\} > 1 - \delta. 
\end{aligned}
\end{equation*}
Finally, thanks to \eqref{eq:HS-min-not-att}, we have
\begin{equation*}
\begin{aligned}
    &\forall \delta > 0, \quad \exists N_s, N_t: \mathbb{P} \biggl\{\|B_{N} - T\|_{HS} - \|T_{N} - T\|_{HS} > \varepsilon_{max} \biggr\} > 1 - \delta.
\end{aligned}
\end{equation*}

\end{appendices}

\printbibliography

\end{document}